\documentclass[11pt,reqno]{amsart}

\usepackage[foot]{amsaddr}
\usepackage{amsmath,amssymb,amsthm}
\usepackage{hyperref}
\usepackage{booktabs}
\usepackage{multirow}
\usepackage{enumerate}
\usepackage{bm}
\usepackage{tikz}
\usetikzlibrary{arrows,shapes}
\tikzset{>=stealth}
\usepackage{float}
\floatstyle{ruled}
\newfloat{algorithm}{htb}{alg}
\floatname{algorithm}{Algorithm}

\usepackage[
maxbibnames=99,
style=numeric,
backend=bibtex,
sorting=nyt,
sortcites,
natbib=true,
doi=false,
url=false,
isbn=false,
eprint=false]{biblatex}

\theoremstyle{plain}
\newtheorem{lemma}{Lemma}
\newtheorem{theorem}{Theorem}

\theoremstyle{definition}
\newtheorem{definition}{Definition}
\newtheorem{example}{Example}
\newtheorem{problem}{Problem}

\theoremstyle{remark}

\newcommand{\reals}{\mathbb{R}}
\newcommand{\ints}{\mathbb{Z}}

\newcommand{\nats}{\mathbb{N}}
\newcommand{\vect}[1]{\bm{#1}}

\DeclareMathOperator{\conv}{conv}

\DeclareMathOperator{\LB}{LB}

\DeclareMathOperator{\BQP}{BQP}
\DeclareMathOperator{\cost}{cost}

\newcommand{\QP}{\mathit{BQP}}

\let\leq\leqslant
\let\geq\geqslant
\let\eps\varepsilon

\title[Convex hulls for graphs of quadratic functions]{Convex hulls for graphs of quadratic functions with unit coefficients: even wheels and complete
  split graphs}

\author[Mitchell Harris]{Mitchell Harris$^{1,2}$}
\address{$^1$School of Mathematics and Physics, University of Queensland, Brisbane, QLD 4072, Australia}
\author[Thomas Kalinowski]{Thomas Kalinowski$^{2}$}
\address{$^2$School of Science and Technology, University of New England, Armidale, NSW 2351, Australia}

\email{m.g.harris@uq.edu.au}
\email{tkalinow@une.edu.au}

\date{}


\begin{document}

\begin{abstract}
  We study the convex hull of the graph of a quadratic function $f(\vect x)=\sum_{ij\in E}x_ix_j$,
  where the sum is over the edge set of a graph $G$ with vertex set $\{1,\dots,n\}$. Using an
  approach proposed by Gupte et al. (Discrete Optimization \textbf{36}, 2020, 100569), we
  investigate minimal extended formulations using additional variables $y_{ij}$, $1\leq i<j\leq n$,
  representing the products $x_ix_j$. The basic idea is to identify a set of facets of the
  Boolean Quadric Polytope which is sufficient for characterizing the convex hull for the given
  graph. Our main results are extended formulations for the cases that the underlying graph $G$ is
  either an even wheel or a complete split graph.
\end{abstract}

\maketitle

\section{Introduction}\label{sec:intro}
In global optimization, the construction of convex envelopes for nonconvex functions is a crucial
ingredient in state-of-the-art solvers. As a consequence, there has been a lot of interest in
characterizing the convex hulls of graphs of nonlinear functions; see~\cite{Locatelli2013} for a
comprehensive treatment. The convex hull might be a polytope with a prohibitively large number of
facets, and in \cite{Gupte2020} it has been proposed to use extended formulations to obtain more
compact representations of such convex hulls. The basic idea is to represent the polytope as a
projection of a higher dimensional polytope with significantly fewer facets, an approach that has
been highly successful in combinatorial optimization~\cite{Conforti2010}. In this paper, we focus on
functions $f : [0,1]^n \to \reals$ of the form
\[f(\vect{x}) = \sum_{ij\in E} x_ix_j\]
where $G=(V,E)$ is a graph with vertex set
$V=[n]:=\{1,2,\dots,n\}$. The convex hull of the graph of $f$ is the set
\[X(f) := \conv \{ (\vect{x},z) \in [0,1]^n \times \reals : \ z = f(\vect{x}) \}.\]
The set $X(f)$ is a polytope, as it is the convex hull of the $2^n$ points $(\vect x,z)$ with $\vect
x\in\{0,1\}^n$ (see \cite{Rikun1997,Sherali1997}). A natural
setting for an extended formulation is to introduce additional variables $y_{ij}$ representing the
products $x_{i}x_{j}$ of two original variables. For $x_i,x_j\in\{0,1\}$, the equation
$y_{ij}=x_ix_j$ is equivalent to the classical McCormick inequalities~\cite{McCormick1976}:
$y_{ij}\geq 0$, $y_{ij}\leq x_i$, $y_{ij} \leq x_j$, and $x_i+x_j-y_{ij}\leq 1$. The \emph{Boolean
  quadric polytope}
\begin{multline*}
  \QP_n = \conv \{ (\vect{x}, \vect{y}) \in \{ 0, 1 \}^{n(n+1)/2}\,:\\
  y_{ij}\geq 0,\,y_{ij}\leq x_i,\,y_{ij}\leq x_j,\,x_i+x_j-y_{ij}\leq 1 \text{ for all } 1 \leq i < j \leq n
   \}
\end{multline*}
was introduced by Padberg~\cite{Padberg1989} in the context of quadratic $0$-$1$-programming, and
its structure has been extensively studied~\cite{Burer2009,Boros1992,Boros1993,Yajima1998}. In our
setting, $\QP_n$ provides an extended formulation for $X(f)$ in the sense that $X(f)=\pi[f](\QP_n)$,
where the map $\pi[f]:[0,1]^{n(n+1)/2}\to[0,1]^n\times\reals$ is defined by
\[\pi[f](\vect{x},\vect{y})=\left(\vect{x},\sum_{ij\in E}y_{ij}\right).\]
As the number of facets of $\QP_n$ is exponential in $n$, the complete list of facets is
known only for small values of $n$, and some of the known facets are NP-hard to
separate~\cite{Padberg1989,Deza1997,Letchford2014,Barahona1986}, the question arises if there exists
a relaxation $P$ of $\BQP_n$, obtained by selecting certain families of facets, such that $\pi[f](P)=X(f)$.
Let $M_n$ be the relaxation obtained by just keeping the McCormick inequalities: 
\[M_n = \left \{ (\vect{x}, \vect{y}) \in [0, 1]^{n(n+1)/2}\,:\,y_{ij}\geq 0,\,y_{ij}\leq
    x_i,\,y_{ij}\leq x_j,\,x_i+x_j-y_{ij}\leq 1 \text{ for all } 1 \leq i < j \leq n \right \}.\]
Then $\pi[f](M_n)\supseteq X(f)$ with equality if and only if $G$ is
bipartite~\cite{Guenluek2012,Luedtke2012}, which has been generalized to functions $\sum_{ij}a_{ij}x_ix_j$ with
arbitrary coefficients $a_{ij}\in\reals$ in~\cite{Misener2015,Boland2017}. Padberg also studied variants of these
polytopes where the variables $y_{ij}$ are only introduced for $ij\in E$:
\begin{align*}
    \QP_n(G) &= \conv \{ (\vect{x}, \vect{y}) \in \{ 0, 1 \}^{n+\lvert E\rvert}\,:\,y_{ij}\geq 0,\,y_{ij}\leq x_i,\,y_{ij}\leq x_j,\,x_i+x_j-y_{ij}\leq 1 \text{ for all
                                                                                       } ij\in
                                                                                       E\},\\
  M_n(G) &= \{ (\vect{x}, \vect{y}) \in [0, 1]^{n+\lvert E\rvert}\,:\,y_{ij}\geq
    0,\,y_{ij}\leq x_i,\,y_{ij}\leq x_j,\,x_i+x_j-y_{ij}\leq 1 \text{ for all } ij\in E\},
\end{align*}
for which we also have $\pi[f](M_n(G))\supseteq\pi[f](\QP_n(G))=X(f)$. In general, the McCormick
relaxation can be quite weak~\cite{Boland2017}, and in~\cite{Gupte2020} the problem was raised to
find relaxations $P$ with $\pi[f](P)=X(f)$. This can be viewed as a weak version of the problem of
characterizing the facets of $\QP_n(G)$ for certain classes of graphs $G$ which has been studied
extensively~\cite{Padberg1989,Simone1990,Michini2018,Boros1992,Bonami2018}. As $\pi[f](P)=X(f)$ is a
weaker condition than $P=\QP_n(G)$, we hope to need fewer facets to achieve this. It is easy to
check (and has been explicitly proved in~\cite{Meyer2005}) that for every $\vect x\in[0,1]^n$,
\begin{multline*}
  \max\{z\,:\,(\vect x,z)\in X(f)\}=\sum_{ij\in E}\min\{x_i,x_j\}\\
  =\max\left\{\sum_{ij\in
      E}y_{ij}\,:\,y_{ij}\leq x_i,\,y_{ij}\leq x_j\text{ for all }1\leq i<j\leq n\right\}.
\end{multline*}
In other words, the McCormick inequalities are sufficient for characterizing the upper boundary of
$X(f)$, and we can focus on the lower boundary. For a polytope $P\subseteq[0,1]^{n+\lvert E\rvert}$
or $P\subseteq[0,1]^{n(n+1)/2}$, let
\[\LB_P(\vect x)=\min\left\{\sum_{ij\in E}y_{ij}\,:\,(\vect x,\,\vect y)\in P\right\}.\]
Then $\pi[f](P)=X(f)$ if and only if, for every $\vect x\in[0,1]$,
\[\{z\,:\,(\vect x,z)\in X(f)\}=\left[\LB_P(\vect x),\,\sum_{ij\in E}\min\{x_i,x_j\}\right].\]
Combining this observation with a technique from~\cite{Zuckerberg2016}, a characterization of
polytopes $P$ with $\pi[f](P)=X(f)$ has been proved in~\cite{Gupte2020}, which is based on certain
subsets of the half-open unit interval $[0,1)$. To state this criterion, we define $\mathcal L$ to
be the collection of unions of finitely many half-open intervals and $\mu$ the Lebesgue measure
(restricted to $\mathcal L$), that is,
  \begin{align*}
    \mathcal L &= \left\{[a_1,b_1)\cup\dots\cup[a_k,b_k)\ :\ 0\leq a_1<b_1<a_2<b_2<\dotsb<a_{k}<b_k\leq 1,\
        k\in\nats\right\},\\
    \mu(X) &= (b_1-a_1)+\dots+(b_k-a_k)\qquad\qquad \text{for }X=[a_1,b_1)\cup\dotsb\cup[a_k,b_k)\in\mathcal L.
  \end{align*} 
\begin{theorem}[\cite{Gupte2020}]\label{thm:characterisation}
  Let $P\subseteq[0,1]^{n+\lvert E\rvert}$ or $P\subseteq[0,1]^{n(n+1)/2}$ be a polytope satisfying
  the following two conditions:
  \begin{enumerate}
  \item For all $(\vect x,\vect y)\in P$ and all $ij\in E$, $y_{ij}\leq\min\{x_i,x_i\}$.
  \item For all $\vect x\in[0,1]^n$, there exists $(\vect x,\vect y)\in P$ with
    $y_{ij}=\min\{x_i,x_j\}$ for all $ij\in E$.    
  \end{enumerate}
  Then $\pi[f](P)=X(f)$ if and only if, for every $\vect x\in[0,1]^n$
  there exist sets $X_1,\dots,X_n\in\mathcal L$ such that $\mu(X_i)=x_i$ for all $i\in[n]$, and
  $\LB_P(\vect x)\geq\sum_{ij\in E}\mu(X_i\cap X_j)$.
\end{theorem}
We illustrate Theorem~\ref{thm:characterisation} in the following example.
\begin{example}\label{ex:characterisation_theorem}
  For $n=3$ and $f(x_1,x_2,x_3)=x_1x_2+x_2x_3+x_1x_3$, the graph $G$ is a triangle. Suppose $P=M_3$
  is the polytope defined by the McCormick inequalities. Then $(\vect x,\vect y)\in P$ where
  $\vect x=(1/2,\,1/2,\,1/2)$ and $\vect y=\vect 0$, hence $\LB_P(\vect x)=0$. By
  Theorem~\ref{thm:characterisation}, if $\pi[f](P)= X(f)$, then there are three pairwise disjoint
  subsets of $[0,1)$, each of them having measure $1/2$. This is impossible, and we conclude that
  $P$ is too weak. To obtain a polytope $P$ with $\pi[f](P)=X(f)$ it is sufficient to add the triangle
  inequality $y_{12}+y_{23}+y_{13}\geq x_1+x_2+x_3-1$. This can be easily shown using the
  characterization in Theorem~\ref{thm:characterisation}: pick $\vect x\in[0,1)^3$ and write down
  corresponding sets $X_1$, $X_2$ and $X_3$. Without loss of generality, $x_1\geq x_2\geq x_3$.
  \begin{itemize}
  \item If $x_1+x_2+x_3\leq 1$, then
    \begin{align*}
      X_1&=[0,x_1), & X_2 &= [x_1,x_1+x_2), & X_3 &= [x_1+x_2,x_1+x_2+x_3),
    \end{align*}
    and for every $\vect y$ with $(\vect x,\vect y)\in P$,
    \[\mu(X_1\cap X_2)+\mu(X_2\cap X_3)+\mu(X_1\cap X_3)=0\leq y_{12}+y_{23}+y_{13}.\]
  \item If $x_1+x_2\leq 1<x_1+x_2+x_3$ then
    \begin{align*}
      X_1 &= [0,x_1), & X_2 &= [x_1,x_1+x_2), & X_3 &= [x_1+x_2,1)\cup[0,x_1+x_2+x_3-1),
    \end{align*}
    and for every $\vect y$ with $(\vect x,\vect y)\in P$,
    \[\mu(X_1\cap X_2)+\mu(X_2\cap X_3)+\mu(X_1\cap X_3)=x_1+x_2+x_3-1\leq y_{12}+y_{23}+y_{13}\]
    due to the triangle inequality.
  \item If $x_1+x_2>1$ and $x_1+x_2+x_3\leq 2$ then
    \begin{align*}
      X_1 &= [0,x_1), & X_2 &= [x_1,1)\cup[0,x_1+x_2-1), & X_3 &= [x_1+x_2-1,x_1+x_2+x_3-1),
    \end{align*}
    and for every $\vect y$ with $(\vect x,\vect y)\in P$,
    \[\mu(X_1\cap X_2)+\mu(X_2\cap X_3)+\mu(X_1\cap X_3)=x_1+x_2+x_3-1\leq y_{12}+y_{23}+y_{13}\]
    due to the triangle inequality.
  \item If $x_1+x_2+x_3>2$ then
    \begin{align*}
      X_1 &= [0,x_1), & X_2 &= [x_1,1)\cup[0,x_1+x_2-1), & X_3 &= [x_1+x_2-1,1)\cup[0,x_1+x_2+x_3-2), 
    \end{align*}
    and for every $\vect y$ with $(\vect x,\vect y)\in P$,
    \begin{multline*}
      \mu(X_1\cap X_2)+\mu(X_2\cap X_3)+\mu(X_1\cap X_3)=1+2(x_1+x_2+x_3-2)\\
      =(x_1+x_2-1)+(x_2+x_3-1)+(x_1+x_3-1)\leq y_{12}+y_{23}+y_{13}
    \end{multline*}
    due to the McCormick inequalities.
  \end{itemize}
\end{example}

One of the classes of facets for $\QP_n$ discussed in~\cite{Padberg1989} is described by the \emph{triangle inequalities}
\begin{align*}
  x_i + x_j + x_k - y_{ij} - y_{ik} - y_{jk} & \leq 1, \\
  - x_i + y_{ij} + y_{ik} - y_{jk} & \leq 0, \\
  - x_j + y_{ij} - y_{ik} + y_{jk} & \leq 0,  \\
  - x_k - y_{ij} + y_{ik} + y_{jk} & \leq 0.
\end{align*}
The strength of the relaxation obtained from the McCormick inequalities together with the triangle
inequalities has been studied both theoretically~\cite{Boros1992} and
computationally~\cite{Bonami2018,Anstreicher2012}. We focus on the first of the triangle inequalities.
\begin{definition}\label{def:triangle_relaxation}
  For a graph $G=(V,E)$ let $T(G)\subseteq[0,1]^{\lvert V\rvert+\lvert E\rvert}$ be the polytope
  determined by the McCormick inequalities for each $ij\in E$ together with the inequalities
  $y_{ij}+y_{jk}+y_{jk}\geq x_i+x_j+x_k-1$ for every triangle $ijk$ in $G$.
\end{definition}
As mentioned above, the McCormick inequalities are sufficient for bipartite graphs, and it would be
interesting to characterize the graphs $G$ such that for the corresponding function $f$,
$\pi[f](T(G))=X(f)$. Our first result provides a first step in this direction by showing that even
wheels belong to this class. Let $W_{n-1}$ be the graph with vertex set $V=[n]$ and edge set
$E=E_0\cup E_1$ where
\begin{align*}
  E_0 &= \{ij\,:\,i\in[n-1],\,j\equiv i+1\pmod{n-1}\}, & E_1 &= \{in\,:\,i\in[n-1]\}.
\end{align*}
\begin{theorem}\label{thm:even_wheels}
  Suppose $n-1$ is even, and let $P=T(W_{n-1})\subseteq[0,1]^{3n-2}$. Then $\pi[f](P)=X(f)$. 
\end{theorem}
Since the wheel graphs have tree-width 3, it is known that $\QP_n(W_{n-1})$ has an extended
formulation with $O(n)$ variables and constraints~\cite{Laurent2009}, but it is still interesting to
have an explicit construction in terms of the natural variables corresponding to vertices and edges.  

Our second result is a generalization of Theorem 1 from~\cite{Gupte2020}, which dealt with the case
that $G$ is a complete graph with one edge removed. We extend this result to the class of complete
split graphs. Let $n=n_1+n_2$, and let $G=(V,E)$ be the graph with vertex set
$V=[n]$, where $V_1=[n_1]$ is a clique, $V_2=V\setminus V_1$ is an independent set, and every vertex
of $V_1$ is adjacent to every vertex of $V_2$, that is,
\[E=\{ij\,:\,1\leq i<j\leq n_1\}\cup\{ij\,:\,1\leq i\leq n_1,\,n_1+1\leq j\leq n\}.\]

For convenience, we denote by $E^*(W)$ for a vertex set
$W\subseteq V$ the set of all pairs of vertices in $W$ (edges and non-edges), that is,
\[E^*(W)=\{ij\,:\,i,j\in W,\,i<j\}.\] We also use the notation $x(W)=\sum_{i\in W}x_i$ and
$y(F)=\sum_{ij\in F}y_{ij}$ for sets $W\subseteq V$ and $F\subseteq E^*(V)$. The clique inequalities
\[y(E^*(W))\geq\alpha x(W)-\binom{\alpha+1}{2}\]
were also introduced in\cite{Padberg1989}. They are valid inequalities for every $W\subseteq V$, $\lvert W\rvert\geq 2$ and every
$\alpha\in\{1,2,\dots,\vert W\rvert-1\}$, and facet-defining for $\lvert W\rvert\geq 3$ and
$1\leq\alpha\leq\lvert W\rvert-2$. For the complete split graphs we need a certain subset of these
split inequalities.
\begin{theorem}\label{thm:complete_split_graphs}
  If $f$ corresponds to a complete split graph with clique $V_1$ and independent set $V_2$, then
  $\pi[f](P)=X(f)$ where $P$ is the polytope described by
  \begin{align*}
  y_{ij} &\geq x_i+x_j-1 && i\in V_1,\,j\in V_2,\\
  y_{ij} &\leq \min\{x_i,\,x_j\} && 1\leq i<j\leq n,\\
y(E^*(V_1\cup S)) &\geq\alpha x(V_1\cup S)-\binom{\alpha+1}{2} && S\subseteq V_2,\,0\leq\lvert S\rvert\leq n_1-1,\,1\leq
                                                        \alpha\leq n_1-1.
  \end{align*}
\end{theorem}
The proofs of Theorems~\ref{thm:even_wheels} and~\ref{thm:complete_split_graphs} are contained in
Sections \ref{sec:wheels} and~\ref{sec:split_graphs}, respectively. In both cases, the the sets
$X_i$ are constructed greedily: The sets are specified one-by-one, and when fixing $X_i$ we aim to
minimize its intersection with the sets $X_j$, $ij\in E$, which are already fixed. This is a natural
strategy, but proving $\LB_P(\vect x)\geq\sum_{ij\in E}\mu(X_i\cap X_j)$ requires a surprising
amount of work. In Section~\ref{sec:conclusion}, we conclude the paper by stating a couple of open
problems.

\section{Even wheels: Proof of Theorem~\ref{thm:even_wheels}}\label{sec:wheels}
In this section, indices $i$ and $j$ are from the set $[n-1]$, and indices like $i+1$ or $j+1$ are
from $[n-1]$ as well, and have to be interpreted modulo $n-1$. Fix an arbitrary
$\vect x\in[0,1]^n$. In view of Theorem~\ref{thm:characterisation}, all we need to do is to find
sets $X_i\subseteq[0,1]$ with $\mu(X_i)=x_i$ and
\[\LB_P(\vect x)\geq\sum_{i=1}^{n-1}\left(\mu(X_i\cap X_{i+1})+\mu(X_i\cap X_n)\right).\]
As a first step we write down a lower bound for $\LB_P(\vect x)$. For a subset $T\subseteq[n-1]$, we
define
\begin{multline*}
  \phi(T)=\sum_{i\in[n-1]\setminus
    T}\max\{0,x_i+x_{i+1}-1\}+\sum_{i\in[n-1]\setminus(T\cup(T+1))}\max\{0,x_i+x_n-1\}\\ +\sum_{i\in T}\max\{0,x_i+x_{i+1}+x_n-1\},
\end{multline*}
and we set $\Phi^*=\max\left\{\phi(T)\,:\,T\subseteq[n-1],\,T\cap(T+1)=\emptyset\right\}$. Subsets
$T\subseteq[n-1]$ with $T\cap(T+1)=\emptyset$ can be identified with certain feasible solutions for
the dual to the linear program defining $\LB_P(\vect x)$, such that $\phi(T)$ is the dual objective
value. As a consequence, $\Phi^*$ is a lower bound for $\LB_P(\vect x)$, and this is the content of
the following lemma. 
\begin{lemma}\label{lem:alternating_triangles}
  $\LB_P(\vect x)\geq\Phi^*$.
\end{lemma}
\begin{proof}
  Let $T\subseteq[n-1]$ with $\phi(T)=\Phi^*$ and $T\cap(T+1)=\emptyset$. Using the partition
  \begin{multline*}
    E=\{\{i,i+1\},\{i,n\},\{i+1,n\}\,:\,i\in T\}\ \cup\ \{\{i,i+1\}\,:\,i\in[n-1]\setminus
    T\}\\ \cup\{\{i,n\}\,:\,i\in[n-1]\setminus \left(T\cup(T+1)\right)\}
  \end{multline*} 
  we have that for every $\vect y$ with $(\vect x,\vect y)\in P$,
  \begin{multline*}
    \sum_{ij\in E}y_{ij}=\sum_{i\in
      T}\left(y_{i,i+1}+y_{i,n}+y_{i+1,n}\right)+\sum_{i\in[n-1]\setminus
      T}y_{i,i+1}+\sum_{i\in[n-1]\setminus \left(T\cup(T+1)\right)}y_{i,n}\\
    \geq\sum_{i\in T}\max\{0,x_i+x_{i+1}+x_n-1\}+\sum_{i\in[n-1]\setminus
      T}\max\{0,x_i+x_{i+1}-1\}\\
    +\sum_{i\in[n-1]\setminus \left(T\cup(T+1)\right)}\max\{0,x_i+x_n-1\}=\phi(T).\qedhere
  \end{multline*}
\end{proof}
In the next lemma, we state a property of optimal sets $T$ which will be useful in the
subsequent arguments.
\begin{lemma}\label{lem:necessary}
  Let $T\subseteq[n-1]$ with $T\cap(T+1)=\emptyset$ and $\phi(T)=\Phi^*$. Then $x_i+x_{i+1}+x_n\leq 2$ for all $i\in T$.
\end{lemma}
\begin{proof}
  Suppose $i\in T$ with $x_i+x_{i+1}+x_n>2$. Then $\min\{x_i+x_{i+1},x_i+x_n,x_{i+1}+x_n\}>1$, and therefore
  \[\phi(T\setminus\{i\})-\phi(T)=\left(2x_1+2x_2+2x_n-3\right)-\left(x_1+x_2+x_n-1\right)=x_1+x_2+x_n-2>0,\]
  which contradicts $\phi(T)=\Phi^*$.
\end{proof}
It will be convenient to assume that $T$ is a subset $\{i\in[n-1]\,:\,1\leq x_i+x_{i+1}+x_n\leq
2\}$, and that it is a maximal subset subject to $T\cap(T+1)=\emptyset$. The next lemma says that
there exists such a $T$ which also maximizes $\phi(T)$.
\begin{lemma}
  There exists a set $T\subseteq[n-1]$ with $T\cap(T+1)=\emptyset$ and $\phi(T)=\Phi^*$ such
  that the following two conditions are satisfied:
  \begin{enumerate}[(i)]
  \item $x_i+x_{i+1}+x_n\geq 1$ for all $i\in T$, and
  \item $1\leq x_i+x_{i+1}+x_n\leq 2\implies \{i-1,i,i+1\}\cap T^*\neq\emptyset$ for all $i\in[n-1]$.
  \end{enumerate}
\end{lemma}
\begin{proof}
  If $x_i+x_{i+1}+x_n<1$ then $\max\{x_i+x_{i+1},x_i+x_n,x_{i+1}+x_n\}<1$, and therefore
  $\phi(T\setminus\{i\})=\phi(T)$. If $1\leq x_i+x_{i+1}+x_n\leq 2$ and $\{i-1,i,i+1\}\cap T=\emptyset$,
  then $T\cup\{i\}$ is a feasible set with $\phi(T\cup\{i\})\geq\phi(T)$. Therefore we can satisfy
  the conditions in the lemma by removing the elements $i\in T$ which violate the first condition,
  and adding the elements $i\in[n-1]\setminus T$ which violate the second condition.
\end{proof}
From now on, we fix a set $T^*\subseteq[n-1]$ with
\begin{enumerate}
\item $T^*\cap(T^*+1)=\emptyset$, and 
\item $\phi(T^*)=\Phi^*$, and
\item $1\leq x_i+x_{i+1}+x_n\leq 2$ for all $i\in T$, and
\item $T^*\cap\{i-1,i,i+1\}\neq\emptyset$ for all $i\in[n-1]$ with $1\leq x_i+x_{i+1}+x_n\leq 2$.
\end{enumerate}
In view of Lemma~\ref{lem:alternating_triangles}, it is sufficient to find sets $X_i\in\mathcal L$
with $\mu(X_i)=x_i$ and
\begin{equation}\label{eq:target}
  \sum_{i=1}^{n-1}\left(\mu(X_i\cap X_{i+1})+\mu(X_i\cap X_n)\right)=\phi(T^*).
\end{equation}
\begin{lemma}\label{lem:squeeze_Xi}
  Let $X_i\subseteq[0,1)$, $i\in[n-1]$, with $\mu(X_i)=x_i$. Then (\ref{eq:target}) is true if and
  only if the following three conditions are satisfied:
  \begin{enumerate}[(i)]
  \item $\mu(X_i\cap X_{i+1})+\mu(X_i\cap X_{n})+\mu(X_{i+1}\cap X_{n})=\max\{0,x_i+x_{i+1}+x_n-1\}$ for every
    $i\in T^*$, 
  \item $\mu(X_i\cap X_{i+1})=\max\{0,x_i+x_{i+1}-1\}$ for all $i\in[n-1]\setminus T^*$, and
  \item $\mu(X_i\cap X_{n})=\max\{0,x_i+x_n-1\}$ for all $i\in[n-1]\setminus \left(T^*\cup(T^*+1)\right)$.
  \end{enumerate}
\end{lemma}
\begin{proof}
  The left-hand side of~(\ref{eq:target}) can
  be expanded as follows:
  \begin{multline*}
    \sum_{i=1}^{n-1}\left(\mu(X_i\cap X_{i+1})+\mu(X_i\cap X_n)\right) = \sum_{i\in
      T}\left(\mu(X_i\cap X_{i+1})+\mu(X_i\cap X_n)+\mu(X_{i+1}\cap X_n)\right)\\
    +\sum_{i\in[n-1]\setminus T^*}\mu(X_i\cap X_{i+1})+\sum_{i\in[n-1]\setminus
      (T^*\cup(T^*+1))}\mu(X_i\cap X_{n}).
  \end{multline*}
  This implies immediately that the three conditions in the lemma are sufficient. To see that they
  are also necessary, we observe that by the inclusion-exclusion principle, for every
  $i\in[n-1]$ and every $j\in\{i+1,n\}$,
  \[\mu(X_i\cap X_j) = \mu(X_i)+\mu(X_j)-\mu(X_i\cup X_j)\geq x_i+x_j-1,\]
  and for every $i\in T^*$,
  \begin{multline*}
    \mu(X_i\cap X_{i+1})+\mu(X_i\cap X_n)+\mu(X_{i+1}\cap X_{n}) \\
    = \mu(X_i)+\mu(X_{i+1})+\mu(X_n)+\mu(X_i\cap
    X_{i+1}\cap X_n)-\mu(X_i\cup X_{i+1}\cup X_n)\geq x_i+x_{i+1}+x_n-1.
  \end{multline*}
  As a consequence,
  \begin{itemize}
  \item $\mu(X_i\cap X_{i+1})+\mu(X_i\cap X_{n})+\mu(X_{i+1}\cap X_{n})\geq\max\{0,x_i+x_{i+1}+x_n-1\}$ for every
    $i\in T^*$, 
  \item $\mu(X_i\cap X_{i+1})\geq\max\{0,x_i+x_{i+1}-1\}$ for all $i\in[n-1]\setminus T^*$, and
  \item $\mu(X_i\cap X_{n})\geq\max\{0,x_i+x_n-1\}$ for all $i\in[n-1]\setminus \left(T^*\cup(T^*+1)\right)$,
  \end{itemize}
  and if any of these inequalities is strict then the left-hand side of~(\ref{eq:target}) is
  strictly larger than the right-hand side.
\end{proof}
For the arguments in the proof of the next lemma it is sometimes convenient to use the following
equivalent statements for the conditions in Lemma~\ref{lem:squeeze_Xi}:
\begin{enumerate}[(i)]
  \item $X_i\cup X_{i+1}\cup X_n=[0,1)$ and $X_i\cap X_{i+1}\cap X_n=\emptyset$ for every
    $i\in T^*$, 
  \item $X_i\cap X_{i+1}=\emptyset$ or $X_i\cup X_{i+1}=[0,1)$ for all $i\in[n-1]\setminus T^*$, and
  \item $X_i\cap X_{n}=\emptyset$ or $X_i\cup X_{n}=[0,1)$ for all $i\in[n-1]\setminus \left(T^*\cup(T^*+1)\right)$.
  \end{enumerate}
In order to state another equivalent condition for the existence of the sets $X_i$, we introduce the notation
\begin{align*}
  m_i &= \max\{0,\,x_i-x_n\}, & M_i &= \min\{x_i,\,1-x_n\},\\
  m'_i &= \min\{1,\,x_i+x_{i+1}\}-x_n, & M'_i &= \max\{1,\,x_i+x_{i+1}\}-x_n,
\end{align*}
for $i\in[n-1]$.
\begin{lemma}\label{lem:equivalence}
  The following two statements are equivalent.
  \begin{enumerate}[(i)]
  \item There exist $X_i\subseteq[0,1]$ with $\mu(X_i)=x_i$ satisfying~(\ref{eq:target}).
  \item There exists a vector $\vect z=(z_1,\dots,z_{n-1})$ which satisfies the system
\begin{align}
  -z_i &\leq -m_i && i\in T^*\cup(T^*+1),\label{ineq:1}\\
  -z_i &\leq -M_i && i\in[n-1]\setminus\left(T^*\cup(T^*+1)\right),\label{ineq:2}\\
  z_i  &\leq M_i && i\in[n-1],\label{ineq:3}\\
  -z_i-z_{i+1} & \leq -M'_i && i\in T^*,\label{ineq:4}\\
  -z_i-z_{i+1} & \leq -m'_i && i\in [n-1]\setminus T^*,\label{ineq:5}\\
  z_i+z_{i+1} &\leq M'_i&& i\in [n-1]\setminus T^*.\label{ineq:6}
\end{align}
  \end{enumerate}
\end{lemma}
\begin{proof}\hfill
  \begin{description}
  \item[``$(i)\Rightarrow(ii)$''] For $i\in[n-1]$, set $z_i=\mu(X_i\setminus X_n)$. Then
    $z_i\geq\mu(X_i)-\mu(X_n)=x_i-x_n$, and together with $z_i\geq 0$, this
    implies~(\ref{ineq:1}). From $X_i\setminus X_n\subseteq X_i$ and
    $X_i\setminus X_n\subseteq[0,1)\setminus X_n$, we obtain $z_i\leq x_i$ and $z_i\leq 1-x_n$,
    hence~(\ref{ineq:3}). The remaining inequalities are obtained by expressing the conditions from
    Lemma~\ref{lem:squeeze_Xi} in terms of the variables $z_i$ in the following way:
    \begin{itemize}
    \item For $i\in T^*$, $X_i\cup X_{i+1}\cup X_n=[0,1)$ implies $(X_i\setminus X_{n-1})\cup
      (X_{i+1}\setminus X_n)\supset[0,1)\setminus X_n$, hence $z_i+z_{i+1}\geq 1-x_n$.
      On the other hand, $X_i\cap X_{i+1}\cap X_n=\emptyset$ implies $\mu(X_i\cap X_n)+\mu(X_{i+1}\cap X_n)\leq\mu(X_n)$, hence
      $(x_i-z_i)+(x_{i+1}-z_{i+1})\leq x_n$. Rearranging the latter inequality gives
      $z_i+z_{i+1}\geq x_i+x_{i+1}-x_n$. Combining these inequalities yields~(\ref{ineq:4}). 
    \item For $i\in[n-1]\setminus T^*$ with $x_i+x_{i+1}\geq 1$, $X_i\cup X_{i+1}=[0,1)$ implies
      \[z_i+z_{i+1}=\mu(X_i\setminus X_n)+\mu(X_{i+1}\setminus X_n)\geq\mu((X_i\cup
        X_{i+1})\setminus X_n)=\mu([0,1)\setminus X_n)=1-x_n,\]
      and 
      \[(x_i-z_i)+(x_{i-1}-z_{i-1})=\mu(X_i\cap X_n)+\mu(X_{i+1}\cap X_n)\geq\mu((X_i\cup
        X_{i+1})\cap X_n)=\mu(X_n)= x_n.\]
      Rearranging the latter inequality gives
      $z_i+z_{i+1}\leq x_i+x_{i+1}-x_n$, and combining the two inequalities, we obtain
      $1-x_n\leq z_i+z_{i+1}\leq x_i+x_{i+1}-x_n$, which is~(\ref{ineq:5}) and~(\ref{ineq:6}) for
      $i\in[n-1]\setminus T^*$ with $x_i+x_{i+1}\geq 1$.
    \item For $i\in[n-1]\setminus T^*$ with $x_i+x_{i+1}\leq 1$, $X_i\cap X_{i+1}=\emptyset$ implies
      \[z_i+z_{i+1}=\mu(X_i\setminus X_n)+\mu(X_{i+1}\setminus X_n)\leq\mu([0,1)\setminus X_n)=1-x_n,\]
      and
      \[(x_i-z_i)+(x_{i-1}-z_{i-1})=\mu(X_i\cap X_n)+\mu(X_{i+1}\cap X_n)\leq\mu(X_n)=x_n.\]
      Rearranging the latter inequality gives
      $z_i+z_{i+1}\geq x_i+x_{i+1}-x_n$, and combining the two inequalities, we obtain
      $x_i+x_{i+1}-x_n\leq z_i+z_{i+1}\leq 1-x_n$, which is~(\ref{ineq:5}) and~(\ref{ineq:6}) for
      $i\in[n-1]\setminus T^*$ with $x_i+x_{i+1}\leq 1$.
    \item For $i\in [n-1]\setminus\left(T^*\cup(T^*+1)\right)$, $X_i\cup X_n=[0,1)$ if $x_i+x_n\geq 1$ and $X_i\cap
      X_n=\emptyset$ if $x_i+x_n\leq 1$. This implies $z_i=\min\{x_i,1-x_n\}$, and in particular,~(\ref{ineq:2}).  
    \end{itemize}
  \item[``$(ii)\Rightarrow(i)$''] We set $X_n=[0,x_n)$, and then
  \[X_i =
    \begin{cases}
      [x_n,x_n+z_i)\cup[0,x_i-z_i) & \text{for odd }i\in[n-1],\\
      [1-z_i,1)\cup[x_n-x_i+z_i,x_n) & \text{for even }i\in[n-1].
    \end{cases}\qedhere
  \] 
  \end{description}
\end{proof}
We complete the proof of Theorem~\ref{thm:even_wheels} by showing that the
system~(\ref{ineq:1})--(\ref{ineq:6}) is feasible. Suppose it isn't. By Farkas' lemma this implies
the existence of non-negative numbers $\pi_i^-$, $\pi_i^+$ and $\sigma_i^-$ for $i\in[n-1]$, and
$\sigma_i^+$ for $i\in[n-1]\setminus T^*$, such that
\begin{align*}
  \pi_i^++\sigma_{i-1}^+ &= \pi_i^-+\sigma_i^-+\sigma_{i-1}^- && i\in T^*,\\
  \pi_i^++\sigma_{i}^+ &= \pi_i^-+\sigma_i^-+\sigma_{i-1}^- && i\in T^*+1,\\
  \pi_i^++\sigma_{i}^++\sigma_{i-1}^+ &= \pi_i^-+\sigma_i^-+\sigma_{i-1}^- && i\in[n-1]\setminus\left(T^*\cup(T^*+1)\right),
\end{align*}
and $h(\vect\pi^-,\vect\pi^+,\vect\sigma^-,\vect\sigma^+)<0$, where
\begin{multline*}
  h(\vect\pi^-,\vect\pi^+,\vect\sigma^-,\vect\sigma^+)=\sum_{i\in
    T^*}\left(-m_i\pi_i^-+M_i\pi_i^+-M_i'\sigma_i^-\right)+\sum_{i\in
    T^*+1}\left(-m_i\pi_i^-+M_i\pi_i^+-m'_i\sigma_i^-+M'_i\sigma_i^+\right)\\
  +\sum_{i\in[n-1]\setminus\left(T^*\cup(T^*+1)\right)}\left(M_i\left(\pi_i^+-\pi_i^-\right)-m'_i\sigma_i^-+M'_i\sigma_i^+\right).
\end{multline*}
This is a negative cost circulation in the network $N(\vect x,T^*)$ with node set $\{O\}\cup[n-1]$, and the arc set
described as follows:
\begin{itemize}
\item There are arcs $(O,i)$ and $(i,O)$ for every $i\in[n-1]$. For $i\in T^*\cup(T^*+1)$, the costs
  are
  \begin{align*}
    \cost(O,i) &=
                 \begin{cases}
                   M_i & \text{if }i\text{ is odd},\\
                   -m_i & \text{if }i\text{ is even},
                 \end{cases}& \cost(i,O) &=
                                           \begin{cases}
                                             -m_i & \text{if }i\text{ is odd},\\
                                             M_i & \text{if }i\text{ is even},
                                           \end{cases}
  \end{align*}
  and for $i\not\in T^*\cup(T^*+1)$,
  \begin{align*}
    \cost(O,i) &=
                 \begin{cases}
                   M_i & \text{if }i\text{ is odd},\\
                   -M_i & \text{if }i\text{ is even},
                 \end{cases}& \cost(i,O) &=
                                           \begin{cases}
                                             -M_i & \text{if }i\text{ is odd},\\
                                             M_i & \text{if }i\text{ is even}.
                                           \end{cases}
  \end{align*}
  The flows on these arcs correspond to the variables $\pi_i^+$ and $\pi_i^-$.
\item For every $i\in T^*$ there is an arc with cost $-M'_i$ corresponding to the variable
  $\sigma_i^-$. This arc is $(i,i+1)$ for odd $i$ and $(i+1,i)$ for even $i$.
\item For every $i\in[n-1]\setminus T^*$ there are two arcs with costs $-m'_i$ and $M'_i$,
  respectively, corresponding to the variables $\sigma_i^-$ and $\sigma^+_i$. For odd $i$, the arc
  with cost $-m'_i$ is $(i,i+1)$ and the arc with cost $M'_i$ is $(i+1,i)$, and for even $i$ it is
  the other way around.
\end{itemize}
The construction is illustrated in Figures~\ref{fig:flow_network_arcs_1}
to~\ref{fig:flow_network_arcs_3}.
\begin{figure}[htb]
  \begin{minipage}{.49\linewidth}
    \centering
    \begin{tikzpicture}
      \node[draw,fill,circle,outer sep=2pt,inner sep=1pt,label={left:{\small $O$}}] (O) at (0,0) {};
      \node[draw,fill,circle,outer sep=2pt,inner sep=1pt,label={right:{\small $i\in T^*$ even}}]
      (i1) at (4,0) {};
      \node[draw,fill,circle,outer sep=2pt,inner sep=1pt,label={right:{\small $i+1$}}] (i2) at (4,2) {};
      \node[draw,fill,circle,outer sep=2pt,inner sep=1pt,label={right:{\small $i-1$}}] (i3) at
      (4,-2) {};
      \draw[thick,->,bend left=10] (O) to node[above] {{\small $-m_i\,(\pi_i^-)$}} (i1);
      \draw[thick,<-,bend left=-10] (O) to node[below] {{\small $M_i\,(\pi_i^+)$}} (i1);
      \draw[thick,->,bend left=-0] (i2) to node[right] {{\small $-M'_i\,(\sigma_i^-)$}} (i1);
      \draw[thick,->,bend left=10] (i3) to node[left] {{\small $-m'_{i-1}\,(\sigma_{i-1}^-)$}} (i1);
      \draw[thick,<-,bend left=-10] (i3) to node[right] {{\small $M'_{i-1}\,(\sigma_{i-1}^+)$}} (i1);
    \end{tikzpicture}
  \end{minipage}\hfill
  \begin{minipage}{.49\linewidth}
    \centering
    \begin{tikzpicture}
      \node[draw,fill,circle,outer sep=2pt,inner sep=1pt,label={left:{\small $O$}}] (O) at (0,0) {};
      \node[draw,fill,circle,outer sep=2pt,inner sep=1pt,label={right:{\small $i\in T^*$ odd}}] (i1) at (4,0)
      {};
      \node[draw,fill,circle,outer sep=2pt,inner sep=1pt,label={right:{\small $i+1$}}] (i2) at (4,2)
      {};
      \node[draw,fill,circle,outer sep=2pt,inner sep=1pt,label={right:{\small $i-1$}}] (i3) at (4,-2)
      {};
      \draw[thick,->,bend left=10] (O) to node[above] {{\small $M_i\,(\pi_i^+)$}} (i1);
      \draw[thick,<-,bend left=-10] (O) to node[below] {{\small $-m_i\,(\pi_i^-)$}} (i1);
      \draw[thick,<-,bend left=-0] (i2) to node[right] {{\small $-M'_i\,(\sigma_i^-)$}} (i1);
      \draw[thick,->,bend left=10] (i3) to node[left] {{\small $M'_{i-1}\,(\sigma_{i-1}^+)$}} (i1);
      \draw[thick,<-,bend left=-10] (i3) to node[right] {{\small $-m'_{i-1}\,(\sigma_{i-1}^-)$}} (i1);
    \end{tikzpicture}
  \end{minipage}
  \caption{The arcs incident with a node $i\in T^*$.}
  \label{fig:flow_network_arcs_1}
\end{figure}
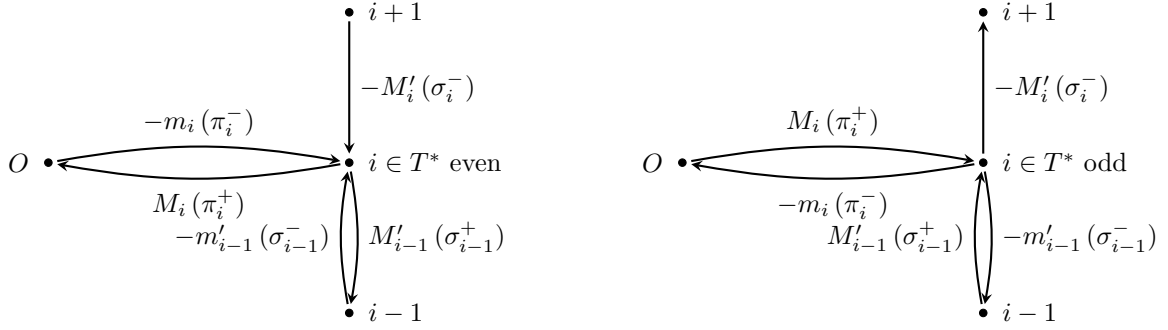
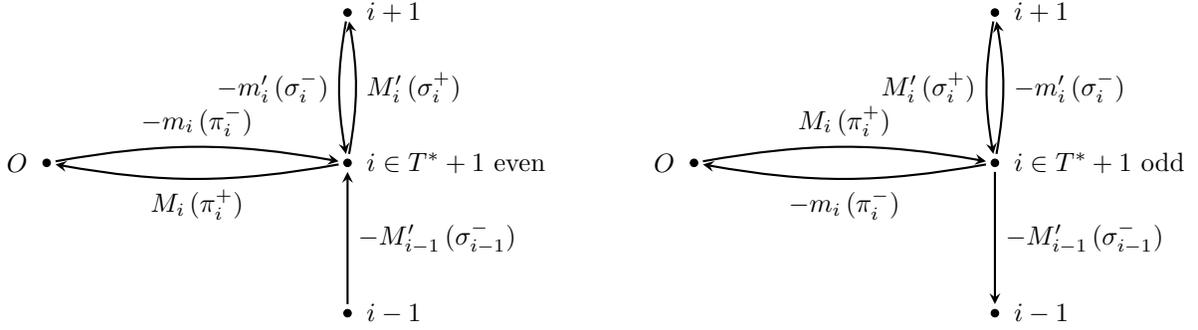
\begin{figure}[htb]
  \begin{minipage}{.49\linewidth}
    \centering
    \begin{tikzpicture}
      \node[draw,fill,circle,outer sep=2pt,inner sep=1pt,label={left:{\small $O$}}] (O) at (0,0) {};
      \node[draw,fill,circle,outer sep=2pt,inner sep=1pt,label={right:{\small $i\in T^*+1$ even}}] (i1) at (4,0)
      {};
      \node[draw,fill,circle,outer sep=2pt,inner sep=1pt,label={right:{\small $i+1$}}] (i2) at (4,2)
      {};
      \node[draw,fill,circle,outer sep=2pt,inner sep=1pt,label={right:{\small $i-1$}}] (i3) at (4,-2)
      {};
      \draw[thick,->,bend left=10] (O) to node[above] {{\small $-m_i\,(\pi_i^-)$}} (i1);
      \draw[thick,<-,bend left=-10] (O) to node[below] {{\small $M_i\,(\pi_i^+)$}} (i1);
      \draw[thick,->,bend left=-10] (i1) to node[right] {{\small $M'_i\,(\sigma_i^+)$}} (i2);
      \draw[thick,<-,bend left=10] (i1) to node[left] {{\small $-m'_i\,(\sigma_i^-)$}} (i2);
      \draw[thick,->,bend left=0] (i3) to node[right] {{\small $-M'_{i-1}\,(\sigma_{i-1}^-)$}} (i1);
    \end{tikzpicture}
  \end{minipage}\hfill
  \begin{minipage}{.49\linewidth}
    \centering
    \begin{tikzpicture}
      \node[draw,fill,circle,outer sep=2pt,inner sep=1pt,label={left:{\small $O$}}] (O) at (0,0) {};
      \node[draw,fill,circle,outer sep=2pt,inner sep=1pt,label={right:{\small $i\in T^*+1$ odd}}] (i1) at (4,0)
      {};
      \node[draw,fill,circle,outer sep=2pt,inner sep=1pt,label={right:{\small $i+1$}}] (i2) at (4,2)
      {};
      \node[draw,fill,circle,outer sep=2pt,inner sep=1pt,label={right:{\small $i-1$}}] (i3) at (4,-2)
      {};
      \draw[thick,->,bend left=10] (O) to node[above] {{\small $M_i\,(\pi_i^+)$}} (i1);
      \draw[thick,<-,bend left=-10] (O) to node[below] {{\small $-m_i\,(\pi_i^-)$}} (i1);
       \draw[thick,->,bend left=-10] (i1) to node[right] {{\small $-m'_i\,(\sigma_i^-)$}} (i2);
      \draw[thick,<-,bend left=10] (i1) to node[left] {{\small $M'_i\,(\sigma_i^+)$}} (i2);
      \draw[thick,<-,bend left=0] (i3) to node[right] {{\small $-M'_{i-1}\,(\sigma_{i-1}^-)$}} (i1);
    \end{tikzpicture}
  \end{minipage}
  \caption{The arcs incident with a node $i\in T^*+1$.}
  \label{fig:flow_network_arcs_2}
\end{figure}
\begin{figure}[htb]
  \begin{minipage}{.49\linewidth}
    \centering
    \begin{tikzpicture}
      \node[draw,fill,circle,outer sep=2pt,inner sep=1pt,label={left:{\small $O$}}] (O) at (0,0) {};
      \node[draw,fill,circle,outer sep=2pt,inner sep=1pt,label={right:{\small $i\not\in T^*\cup(T^*+1)$ even}}] (i1) at (4,0)
      {};
      \node[draw,fill,circle,outer sep=2pt,inner sep=1pt,label={right:{\small $i+1$}}] (i2) at (4,2)
      {};
      \node[draw,fill,circle,outer sep=2pt,inner sep=1pt,label={right:{\small $i-1$}}] (i3) at (4,-2)
      {};
      \draw[thick,->,bend left=10] (O) to node[above] {{\small $-M_i\,(\pi_i^-)$}} (i1);
      \draw[thick,<-,bend left=-10] (O) to node[below] {{\small $M_i\,(\pi_i^+)$}} (i1);
      \draw[thick,->,bend left=-10] (i1) to node[right] {{\small $M'_i\,(\sigma_i^+)$}} (i2);
      \draw[thick,<-,bend left=10] (i1) to node[left] {{\small $-m'_i\,(\sigma_i^-)$}} (i2);
      \draw[thick,->,bend left=0] (i3) to node[right] {{\small $-M'_{i-1}\,(\sigma_{i-1}^-)$}} (i1);
    \end{tikzpicture}
  \end{minipage}\hfill
  \begin{minipage}{.49\linewidth}
    \centering
    \begin{tikzpicture}
      \node[draw,fill,circle,outer sep=2pt,inner sep=1pt,label={left:{\small $O$}}] (O) at (0,0) {};
      \node[draw,fill,circle,outer sep=2pt,inner sep=1pt,label={right:{\small $i\not\in T^*\cup(T^*+1)$ odd}}] (i1) at (4,0)
      {};
      \node[draw,fill,circle,outer sep=2pt,inner sep=1pt,label={right:{\small $i+1$}}] (i2) at (4,2)
      {};
      \node[draw,fill,circle,outer sep=2pt,inner sep=1pt,label={right:{\small $i-1$}}] (i3) at (4,-2)
      {};
      \draw[thick,->,bend left=10] (O) to node[above] {{\small $M_i\,(\pi_i^+)$}} (i1);
      \draw[thick,<-,bend left=-10] (O) to node[below] {{\small $-M_i\,(\pi_i^-)$}} (i1);
       \draw[thick,->,bend left=-10] (i1) to node[right] {{\small $-m'_i\,(\sigma_i^-)$}} (i2);
      \draw[thick,<-,bend left=10] (i1) to node[left] {{\small $M'_i\,(\sigma_i^+)$}} (i2);
      \draw[thick,<-,bend left=0] (i3) to node[right] {{\small $-M'_{i-1}\,(\sigma_{i-1}^-)$}} (i1);
    \end{tikzpicture}
  \end{minipage}
  \caption{The arcs incident with a node $i\in[n-1]\setminus\left(T^*\cup(T^*+1)\right)$.}
  \label{fig:flow_network_arcs_3}
\end{figure}
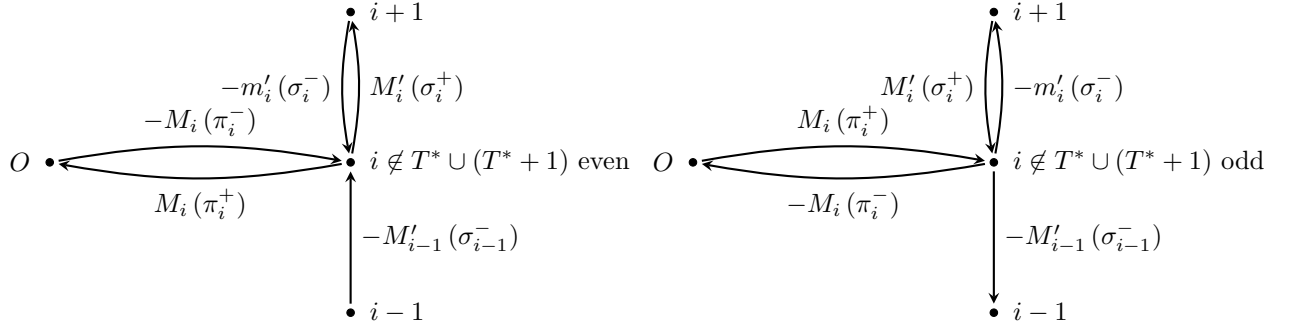

Figure~\ref{fig:flow_network_costs} shows the complete network for $n=13$ and $T^*=\{1,3,6,8,10\}$.
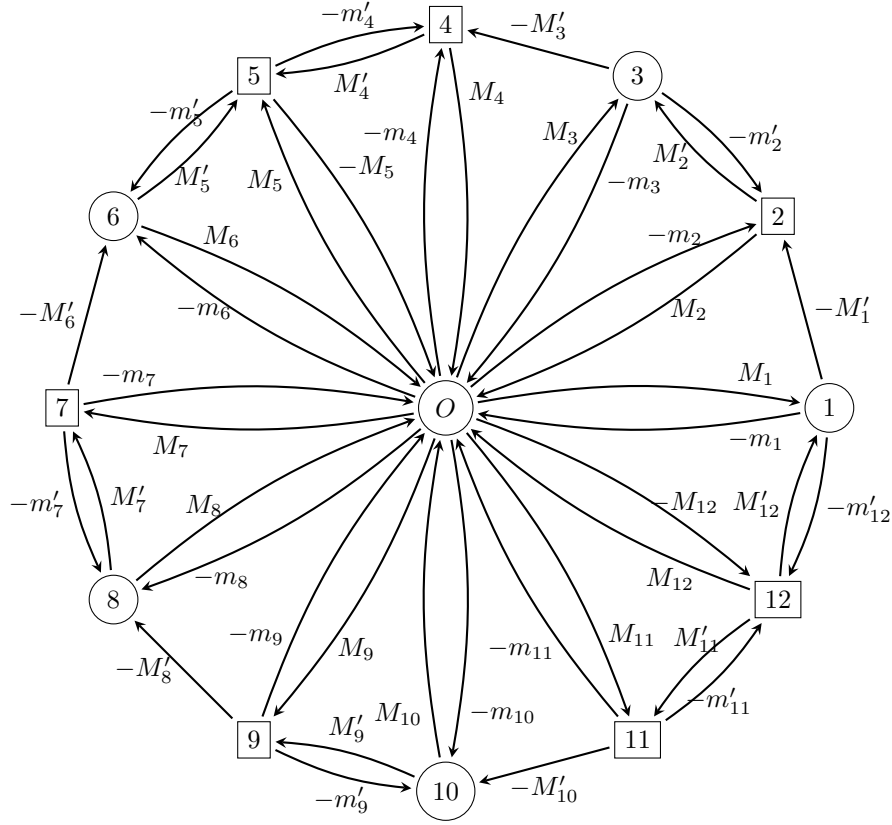
\begin{figure}[htb]
  \centering
  \begin{tikzpicture}[scale=1.7]
    \node[draw,circle,outer sep=2pt] (O) at (0,0) {{\small $O$}};
    \node[draw,circle,outer sep=2pt] (v1) at (0:3) {{\small $1$}};
    \node[draw,rectangle,outer sep=2pt] (v2) at (30:3) {{\small $2$}};
    \node[draw,circle,outer sep=2pt] (v3) at (60:3) {{\small $3$}};
    \node[draw,rectangle,outer sep=2pt] (v4) at (90:3) {{\small $4$}};
    \node[draw,rectangle,outer sep=2pt] (v5) at (120:3) {{\small $5$}};
    \node[draw,circle,outer sep=2pt] (v6) at (150:3) {{\small $6$}};
    \node[draw,rectangle,outer sep=2pt] (v7) at (180:3) {{\small $7$}};
    \node[draw,circle,outer sep=2pt] (v8) at (210:3) {{\small $8$}};
    \node[draw,rectangle,outer sep=2pt] (v9) at (240:3) {{\small $9$}};
    \node[draw,circle,outer sep=2pt] (v10) at (270:3) {{\small $10$}};
    \node[draw,rectangle,outer sep=2pt] (v11) at (300:3) {{\small $11$}};
    \node[draw,rectangle,outer sep=2pt] (v12) at (330:3) {{\small $12$}};
    \draw[thick,->,bend left=10] (O) to node[above,very near end] {{\small $M_1$}} (v1);
    \draw[thick,<-,bend left=-10] (O) to node[below,very near end] {{\small $-m_1$}} (v1);
    \draw[thick,->,bend left=10] (O) to node[above,near end] {{\small $-m_2$}} (v2);
    \draw[thick,<-,bend left=-10] (O) to node[below,near end] {{\small $M_2$}} (v2);
    \draw[thick,->,bend left=10] (O) to node[left,very near end] {{\small $M_3$}} (v3);
    \draw[thick,<-,bend left=-10] (O) to node[right,near end] {{\small $-m_3$}} (v3);
    \draw[thick,->,bend left=10] (O) to node[left,near end] {{\small $-m_4$}} (v4);
    \draw[thick,<-,bend left=-10] (O) to node[right,very near end] {{\small $M_4$}} (v4);
    \draw[thick,->,bend left=10] (O) to node[left,near end] {{\small $M_5$}} (v5);
    \draw[thick,<-,bend left=-10] (O) to node[right,near end] {{\small $-M_5$}} (v5);
    \draw[thick,->,bend left=10] (O) to node[below,near end] {{\small $-m_6$}} (v6);
    \draw[thick,<-,bend left=-10] (O) to node[above,near end] {{\small $M_6$}} (v6);
    \draw[thick,->,bend left=10] (O) to node[below,near end] {{\small $M_7$}} (v7);
    \draw[thick,<-,bend left=-10] (O) to node[above,very near end] {{\small $-m_7$}} (v7);
    \draw[thick,->,bend left=10] (O) to node[below,near end] {{\small $-m_8$}} (v8);
    \draw[thick,<-,bend left=-10] (O) to node[above,near end] {{\small $M_8$}} (v8);
    \draw[thick,->,bend left=10] (O) to node[right,near end] {{\small $M_9$}} (v9);
    \draw[thick,<-,bend left=-10] (O) to node[left,near end] {{\small $-m_9$}} (v9);
    \draw[thick,->,bend left=10] (O) to node[right,very near end] {{\small $-m_{10}$}} (v10);
    \draw[thick,<-,bend left=-10] (O) to node[left,very near end] {{\small $M_{10}$}} (v10);
    \draw[thick,->,bend left=10] (O) to node[right,near end] {{\small $M_{11}$}} (v11);
    \draw[thick,<-,bend left=-10] (O) to node[left,near end] {{\small $-m_{11}$}} (v11);
    \draw[thick,->,bend left=10] (O) to node[above,near end] {{\small $-M_{12}$}} (v12);
    \draw[thick,<-,bend left=-10] (O) to node[below,near end] {{\small $M_{12}$}} (v12);
    \draw[thick,<-] (v2) to node[right] {{\small $-M'_1$}} (v1); \draw[thick,<-] (v4) to node[above]
    {{\small $-M'_3$}} (v3); \draw[thick,<-] (v6) to node[left] {{\small $-M'_6$}} (v7);
    \draw[thick,<-] (v8) to node[left] {{\small $-M'_8$}} (v9); \draw[thick,<-] (v10) to node[below]
    {{\small $-M'_{10}$}} (v11); \draw[thick,->,bend left=10] (v3) to node[right] {{\small
        $-m'_{2}$}} (v2); \draw[thick,<-,bend left=-10] (v3) to node[left] {{\small $M'_{2}$}} (v2);
    \draw[thick,->,bend left=10] (v5) to node[above] {{\small $-m'_{4}$}} (v4); \draw[thick,<-,bend
    left=-10] (v5) to node[below] {{\small $M'_{4}$}} (v4); \draw[thick,<-,bend left=10] (v6) to
    node[above] {{\small $-m'_{5}$}} (v5); \draw[thick,->,bend left=-10] (v6) to node[below]
    {{\small $M'_{5}$}} (v5); \draw[thick,<-,bend left=10] (v8) to node[left] {{\small $-m'_{7}$}}
    (v7); \draw[thick,->,bend left=-10] (v8) to node[right] {{\small $M'_{7}$}} (v7);
    \draw[thick,<-,bend left=10] (v10) to node[below] {{\small $-m'_{9}$}} (v9); \draw[thick,->,bend
    left=-10] (v10) to node[above] {{\small $M'_{9}$}} (v9); \draw[thick,<-,bend left=10] (v12) to
    node[below] {{\small $-m'_{11}$}} (v11); \draw[thick,->,bend left=-10] (v12) to node[above]
    {{\small $M'_{11}$}} (v11); \draw[thick,->,bend left=10] (v1) to node[right] {{\small
        $-m'_{12}$}} (v12); \draw[thick,<-,bend left=-10] (v1) to node[left] {{\small $M'_{12}$}}
    (v12);
  \end{tikzpicture}
  \caption{The flow network for $n=13$ and $T^*=\{1,3,6,8,10\}$.}
  \label{fig:flow_network_costs}
\end{figure}

We conclude the proof of Theorem~\ref{thm:even_wheels} by proving the following lemma.
\begin{lemma}\label{lem:no_negative_cycles}
  The network $N(\vect x,T^*)$ does not contain a directed cycle of negative cost.
\end{lemma}
\begin{proof}
  Suppose there is a negative cost cycle and let $C$ be a negative cost cycle with the minimum
  number of arcs. We will argue that we can use $C$ to modify $T^*$ to obtain a set $T$ with
  $T\cap(T+1)=\emptyset$ and $\phi(T)>\phi(T^*)$, which is the required contradiction. If $C$ is a
  ``backward'' cycle, that is, $C=(1,n-1,n-2,\dots,2,1)$ or $C=(O,i+1,i,i-1,\dots,j,O)$, then in the
  network $N((x_2,x_3,\dots,x_{n-1},x_1,x_n),T^*-1)$, the cycle $C'=(1,2,\dots,n-1,1)$ or
  $C'=(O,i-1,i,\dots,j)$ is a cycle of negative cost. As a consequence, we can assume that $C$ is a
  forward cycle, that is, we are in one of the following two cases, where $[i,j]$ denotes the set
  $\{i,i+1,\dots,j\}$ and the elements of $[n-1]$ are arranged cyclically so that $n-1$ is followed
  by $1$.   
  \begin{description}
  \item[Case 1] $C=(1,2,\dots,n-1)$. Then
    \[T^*=\left\{i\in[n-1]\,:\,i\text{ odd and }1\leq x_i+x_{i+1}+x_n\leq 2\right\},\]
    and we can take $T=\left\{i\in[n-1]\,:\,i\text{ even and }x_i+x_{i+1}+x_n\leq 2\right\}$.
  \item[Case 2] $C=(O,i,i+1,\dots,j,j+1,O)$. Then
   \[T^*\cap[i,j]=\left\{k\in[i,j]\,:\,k\text{ odd and }1\leq x_k+x_{k+1}+x_n\leq 2\right\},\]  and
    we can take
    $T=\left(T^*\setminus[i,j]\right)\cup\left\{k\in[i,j]\,:\,k\text{ even and }1\leq
      x_k+x_{k+1}+x_n\leq 2\right\}$.
    \end{description}
    The basic idea is that $-\cost(C)$ is a lower bound for $\phi(T)-\phi(T^*)$. The verification is
    straightforward but rather tedious, and we have decided to provide the details in
    Appendix~\ref{app:proof_details}.
\end{proof}

\section{Complete split graphs: Proof of Theorem~\ref{thm:complete_split_graphs}}\label{sec:split_graphs}
In this section, we prove that for a complete split graph with clique $V_1$ and independent set
$V_2$, $\pi[f](P)=X(f)$ where $P$ is the polytope described by the following inequalities:
\begin{align}
  y_{ij} &\geq x_i+x_j-1 && i\in V_1,\,j\in V_2,\label{eq:McCormick_LB}\\
  y_{ij} &\leq \min\{x_i,\,x_j\} && 1\leq i<j\leq n,\label{eq:McCormick_UB}\\
y(E^*(V_1\cup S)) &\geq \alpha x(V_1\cup S)-\binom{\alpha+1}{2} && S\subseteq V_2,\,0\leq\lvert S\rvert\leq n_1-1,\,1\leq
                                                        \alpha\leq n_1-1. \label{eq:clique}
\end{align}
We will prove this by providing a construction for sets $X_i$ as required by
Theorem~\ref{thm:characterisation}. To avoid certain case distinctions, we start by showing
that we can assume that $\vect x$ is in the interior of $[0,1]^n$.
\begin{lemma}\label{lem:interior_of_cube}
  Suppose there exists $\vect x\in[0,1]^n$ such that whenever
  $X_1,\dots,X_n\subseteq[0,1]$ satisfy $\mu(X_i)=x_i$ for all $i\in[n]$ then $\LB_P(\vect
  x)<\sum_{ij\in E}\mu(X_i\cap X_j)$. Then there is such a vector $\vect x$ with $0<x_i<1$ for all $i\in[n]$. 
\end{lemma}
\begin{proof}
  For $\vect x\in[0,1]^n$ set
  \[\overline\mu(\vect x)=\min\left\{\sum_{ij\in E}\mu(X_i\cap X_j)\,:\,\mu(X_i)=x_i\text{ for all
      }i\in V\right\}.\]
  Suppose the statement of the lemma is wrong, and let $\vect x\in[0,1]$ be a vector such that
  $\LB_P(\vect x)=\overline\mu(\vect x)-\delta$ with $\delta>0$. Let $\vect x'$ be a vector obtained
  from $\vect x$ by setting
  \[x'_i=
    \begin{cases}
      x_i & \text{if }0<x_i<1,\\
      \eps & \text{if }x_i=0,\\
      1-\eps & \text{if }x_i=1.
    \end{cases}    
  \]
  By assumption $\overline\mu(\vect x')=\LB_P(\vect x')$. Let $X'_1,\dots,X'_n\subset[0,1)$ with
  $\mu(X'_i)=x'_i$ for all $i\in V$ and
  \[\sum_{ij\in E}\mu(X'_i\cap X'_j)=\overline\mu(\vect x')=\LB_P(\vect x').\]
  We define sets $X_1,\dots,X_n$ by setting $X_i=X'_i$ if $0<x_i<1$, $X_i=\emptyset$ if $x_i=0$ and
  $X_i=[0,1)$ if $x_i=1$. For $\eps>0$ sufficiently small,
  \[\overline\mu(\vect x)\leq\sum_{ij\in E}\mu(X_i\cap X_j)\leq\sum_{ij\in E}\mu(X'_i\cap X'_j)+\eps
    n(n-1)=\LB_P(\vect x')+\eps n(n-1)<\LB_P(\vect x')+\delta/2.\]
  The linear programs defining $\LB_P(\vect x)$ and $\LB_P(\vect x')$ differ only in the right-hand
  sides of their constraints, hence by choosing $\eps$ sufficiently small, we can assume that
  $\LB_P(\vect x')<\LB(\vect x)+\delta/2$. This implies $\overline\mu(\vect x)<\LB_P(\vect
  x)+\delta=\overline\mu(\vect x)$, which is the required contradiction. 
\end{proof}
Using Lemma~\ref{lem:interior_of_cube} we can assume without loss of generality that
$1>x_{n_1+1}\geq x_{n_1+2}\geq\dots\geq x_{n}>0$ and $1>x_1\geq x_2\geq\dots\geq x_{n_1}>0$. We set
$X_i=[0,x_i)$ for $i\in V_2$, and construct the sets $X_i$ for $i\in V_1$ greedily. Suppose the sets
$X_{j}$ for $j<i$ have been specified. Then $X_i$ is chosen to minimize
\[\sum_{j=n_1+1}^n\mu(X_i\cap X_j)+\sum_{j=1}^{i-1}\mu(X_i\cap X_j).\]
More precisely, this can be done as described in Algorithm~\ref{alg:construction}.
\begin{algorithm}
  \caption{Construction of the sets $X_i$.}\label{alg:construction}
  \begin{tabbing}
    .....\=.....\=.....\=............................. \kill \\[-1ex]   
    Initialize $(a_0,a_1,\dots,a_{n_2},a_{n_2+1})=(1,x_{n_1+1},x_{n_1+2},\dots,x_{n},0)$\\[1ex]
    \textbf{for} $i=1,2,\dots,n_1$ \textbf{do}\\[1ex]
    \> $p\leftarrow\min\{k\,:\,a_k+x_i<1\}$\\[1ex]
    \> $X_i\leftarrow[a_{p-1},\,1)\cup[a_p,\,a_p+a_{p-1}+x_i-1)$\\[1ex]
    \> \textbf{if} $p=1$ \textbf{then}\\[1ex]
    \> \> $a_{1}\leftarrow a_{1}+x_i$\\[1ex]
    \> \textbf{else}\\[1ex]
    \> \>
    $(a_{p-1},a_{p},\dots,a_{n_2})\leftarrow(a_p+a_{p-1}+x_i-1,\,a_{p+1},\,a_{p+2},\dots,\,a_{n_2+1})$\\[2ex]
    \textbf{return} $X_1,\dots,X_{n_1}$
  \end{tabbing}
\end{algorithm}
\begin{example}\label{ex:split_graph_algorithm}
  Let $n_1=4$, $n_2=5$ and $\vect x=\left(0.85,0.8,0.7,0.5,0.8,0.6,0.5,0.3,0.1\right)$ The steps of
  Algorithm~\ref{alg:construction} are illustrated in Figures~\ref{fig:example_step_0}
  to~\ref{fig:example_step_4}.  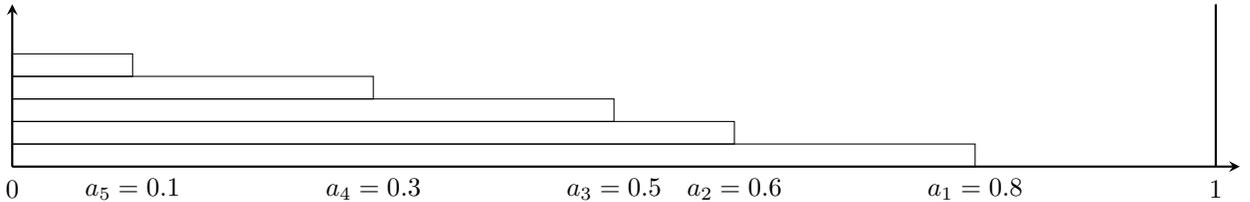
\begin{figure}[htb]
  \centering
  \begin{tikzpicture}[xscale=1.6,yscale=.3]
       \draw[thick,->] (0,0) to (10.2,0);
       \draw[thick,->] (0,0) to (0,7.2);
       \draw[thick] (10,0) to (10,7.2);
       \node at (0,-1) {{\small $0$}};
       \node at (10,-1) {{\small $1$}};
       \draw (0,0) rectangle (8,1);
       \draw (0,1) rectangle (6,2);
       \draw (0,2) rectangle (5,3);
       \draw (0,3) rectangle (3,4);
       \draw (0,4) rectangle (1,5);
       \node at (1,-1) {{\small $a_5=0.1$}};
       \node at (3,-1) {{\small $a_4=0.3$}};
       \node at (5,-1) {{\small $a_3=0.5$}};
       \node at (6,-1) {{\small $a_2=0.6$}};
       \node at (8,-1) {{\small $a_1=0.8$}};
  \end{tikzpicture}
  \caption{$X_5=[0,0.8)$, $X_6=[0,0.6)$, $X_7=[0,0.5)$, $X_8=[0,0.3)$, $X_9=[0,0.1)$.}
  \label{fig:example_step_0}
\end{figure}
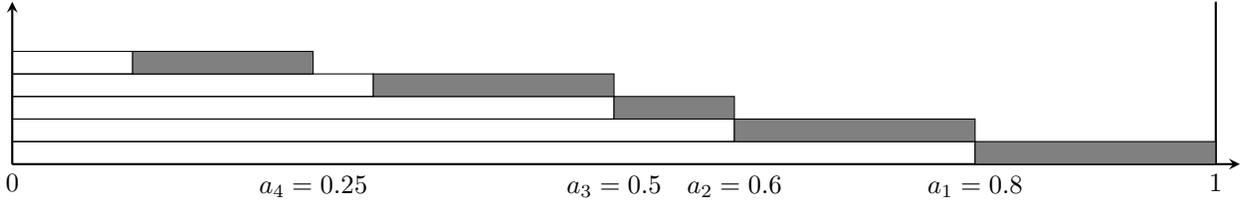
\begin{figure}[htb]
  \centering
  \begin{tikzpicture}[xscale=1.6,yscale=.3]
       \draw[thick,->] (0,0) to (10.2,0);
       \draw[thick,->] (0,0) to (0,7.2);
       \draw[thick] (10,0) to (10,7.2);
       \node at (0,-.8) {{\small $0$}};
       \node at (10,-.8) {{\small $1$}};
       \draw (0,0) rectangle (8,1);
       \draw (0,1) rectangle (6,2);
       \draw (0,2) rectangle (5,3);
       \draw (0,3) rectangle (3,4);
       \draw (0,4) rectangle (1,5);
       \draw[fill=gray] (8,0) rectangle (10,1);
       \draw[fill=gray] (6,1) rectangle (8,2);
       \draw[fill=gray] (5,2) rectangle (6,3);
       \draw[fill=gray] (3,3) rectangle (5,4);
       \draw[fill=gray] (1,4) rectangle (2.5,5);
       \node at (2.5,-1) {{\small $a_4=0.25$}};
       \node at (5,-1) {{\small $a_3=0.5$}};
       \node at (6,-1) {{\small $a_2=0.6$}};
       \node at (8,-1) {{\small $a_1=0.8$}};
  \end{tikzpicture}
  \caption{$X_1=[0.3,1)\cup[0.1,0.25)$.}
  \label{fig:example_step_1}
\end{figure}
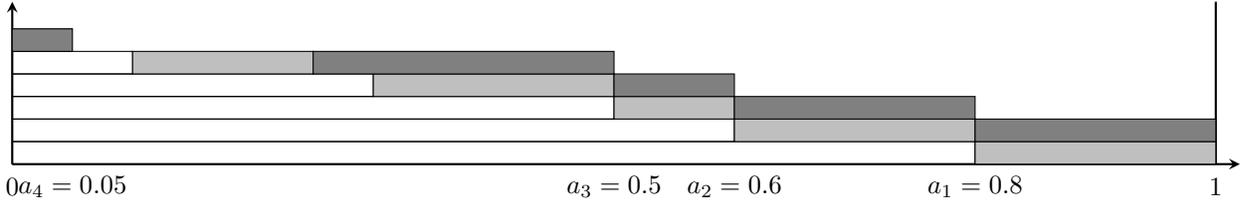
\begin{figure}[htb]
  \centering
  \begin{tikzpicture}[xscale=1.6,yscale=.3]
       \draw[thick,->] (0,0) to (10.2,0);
       \draw[thick,->] (0,0) to (0,7.2);
       \draw[thick] (10,0) to (10,7.2);
       \node at (0,-1) {{\small $0$}};
       \node at (10,-1) {{\small $1$}};
       \draw (0,0) rectangle (8,1);
       \draw (0,1) rectangle (6,2);
       \draw (0,2) rectangle (5,3);
       \draw (0,3) rectangle (3,4);
       \draw (0,4) rectangle (1,5);
       \draw[fill=lightgray] (8,0) rectangle (10,1);
       \draw[fill=lightgray] (6,1) rectangle (8,2);
       \draw[fill=lightgray] (5,2) rectangle (6,3);
       \draw[fill=lightgray] (3,3) rectangle (5,4);
       \draw[fill=lightgray] (1,4) rectangle (2.5,5);
       \draw[fill=gray] (8,1) rectangle (10,2);
       \draw[fill=gray] (6,2) rectangle (8,3);
       \draw[fill=gray] (5,3) rectangle (6,4);
       \draw[fill=gray] (2.5,4) rectangle (5,5);
       \draw[fill=gray] (0,5) rectangle (0.5,6);
       \node at (.5,-1) {{\small $a_4=0.05$}};
       \node at (5,-1) {{\small $a_3=0.5$}};
       \node at (6,-1) {{\small $a_2=0.6$}};
       \node at (8,-1) {{\small $a_1=0.8$}};
  \end{tikzpicture}
  \caption{$X_2=[0.25,1)\cup[0,0.05)$.}
  \label{fig:example_step_2}
\end{figure}
\begin{figure}[htb]
  \centering
  \begin{tikzpicture}[xscale=1.6,yscale=.3]
       \draw[thick,->] (0,0) to (10.2,0);
       \draw[thick,->] (0,0) to (0,7.2);
       \draw[thick] (10,0) to (10,7.2);
       \node at (0,-1) {{\small $0$}};
       \node at (10,-1) {{\small $1$}};
       \draw (0,0) rectangle (8,1);
       \draw (0,1) rectangle (6,2);
       \draw (0,2) rectangle (5,3);
       \draw (0,3) rectangle (3,4);
       \draw (0,4) rectangle (1,5);
       \draw[fill=lightgray] (8,0) rectangle (10,1);
       \draw[fill=lightgray] (6,1) rectangle (8,2);
       \draw[fill=lightgray] (5,2) rectangle (6,3);
       \draw[fill=lightgray] (3,3) rectangle (5,4);
       \draw[fill=lightgray] (1,4) rectangle (2.5,5);
       \draw[fill=lightgray] (8,1) rectangle (10,2);
       \draw[fill=lightgray] (6,2) rectangle (8,3);
       \draw[fill=lightgray] (5,3) rectangle (6,4);
       \draw[fill=lightgray] (2.5,4) rectangle (5,5);
       \draw[fill=lightgray] (0,5) rectangle (0.5,6);
       \draw[fill=gray] (8,2) rectangle (10,3);
       \draw[fill=gray] (6,3) rectangle (8,4);
       \draw[fill=gray] (5,4) rectangle (6,5);
       \draw[fill=gray] (0.5,5) rectangle (2.5,6);
       \node at (2.5,-1) {{\small $a_3=0.25$}};
       \node at (6,-1) {{\small $a_2=0.6$}};
       \node at (8,-1) {{\small $a_1=0.8$}};
  \end{tikzpicture}
  \caption{$X_3=[0.5,1)\cup[0.05,0.25)$.}
  \label{fig:example_step_3}
\end{figure}
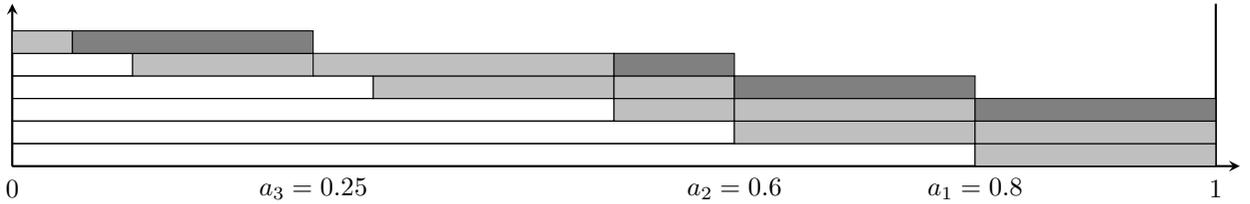
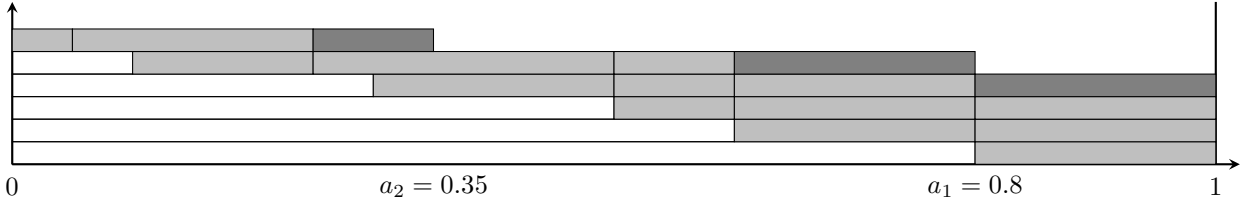
\begin{figure}[htb]
  \centering
  \begin{tikzpicture}[xscale=1.6,yscale=.3]
       \draw[thick,->] (0,0) to (10.2,0);
       \draw[thick,->] (0,0) to (0,7.2);
       \draw[thick] (10,0) to (10,7.2);
       \node at (0,-1) {{\small $0$}};
       \node at (10,-1) {{\small $1$}};
       \draw (0,0) rectangle (8,1);
       \draw (0,1) rectangle (6,2);
       \draw (0,2) rectangle (5,3);
       \draw (0,3) rectangle (3,4);
       \draw (0,4) rectangle (1,5);
       \draw[fill=lightgray] (8,0) rectangle (10,1);
       \draw[fill=lightgray] (6,1) rectangle (8,2);
       \draw[fill=lightgray] (5,2) rectangle (6,3);
       \draw[fill=lightgray] (3,3) rectangle (5,4);
       \draw[fill=lightgray] (1,4) rectangle (2.5,5);
       \draw[fill=lightgray] (8,1) rectangle (10,2);
       \draw[fill=lightgray] (6,2) rectangle (8,3);
       \draw[fill=lightgray] (5,3) rectangle (6,4);
       \draw[fill=lightgray] (2.5,4) rectangle (5,5);
       \draw[fill=lightgray] (0,5) rectangle (0.5,6);
       \draw[fill=lightgray] (8,2) rectangle (10,3);
       \draw[fill=lightgray] (6,3) rectangle (8,4);
       \draw[fill=lightgray] (5,4) rectangle (6,5);
       \draw[fill=lightgray] (0.5,5) rectangle (2.5,6);
       \draw[fill=gray] (8,3) rectangle (10,4);
       \draw[fill=gray] (6,4) rectangle (8,5);
       \draw[fill=gray] (2.5,5) rectangle (3.5,6);         
       \node at (3.5,-1) {{\small $a_2=0.35$}};
       \node at (8,-1) {{\small $a_1=0.8$}};
  \end{tikzpicture}
  \caption{$X_4=[0.6,1)\cup[0.25,0.35)$.}
  \label{fig:example_step_4}
\end{figure}
 From
  Figure~\ref{fig:example_step_4} we see that
  \[\sum_{ij\in E}\mu(X_i\cap X_j) =
    0.35\binom{6}{2}+0.45\binom{5}{2}+0.2\binom{4}{2}-0.6-2\times0.5-3\times0.3-4\times 0.1=8.05.\]
  In order to see that this is a lower bound for $y(E)$, whenever $(\vect x,\vect y)$
  satisfies~(\ref{eq:McCormick_LB}),~(\ref{eq:McCormick_UB}) and~(\ref{eq:clique}), we start
  with~(\ref{eq:clique}) for $S=\{7,8,9\}$ and $t=3$:
  \[y(E^*(V\setminus\{5,6\}))\geq 3\times 3.75-\binom{4}{2}=5.25.\] With~(\ref{eq:McCormick_UB}) for
  the pairs $(i,j)$ with $i,j\in\{7,8,9\}$ and $i<j$, we obtain
  \[y(E(V\setminus\{5,6\}))\geq 5.25-0.3-2\times 0.1=4.75.\] Finally, we add the edges incident with
  nodes $5$ or $6$, using~(\ref{eq:McCormick_LB}) for the lower bound:
  \[y(E(V))\geq 4.75+0.65+0.6+0.5+0.3+0.45+0.4+0.3+0.1=8.05.\]
\end{example}
To prove that the sets $X_i$ constructed by Algorithm~\ref{alg:construction} have the required
properties we will extend the algorithm so that it also produces a set $S\subseteq V_2$ with the
following properties:
\begin{enumerate}
\item $\lvert S\rvert\leq n_1-1$.
\item $\mu(X_i\cap X_j)=\max\{0,x_i+x_j-1\}$ for all $i\in V_1$, $j\in V_2\setminus S$,
\item There exists $\alpha\in\{0,1,\dots,n_1-1\}$ such that for every $t\in[0,1)$,
  \[\left\lvert\{i\in V_1\cup S\,:\,t\in X_i\}\right\rvert\in\{\alpha,\alpha+1\}.\]
\end{enumerate}
Actually, the number $\alpha$ in the third condition is
$\alpha=\left\lfloor x(V_1\cup S)\right\rfloor$. Assuming the existence of such a set $S$, the
theorem is proved as follows.
\begin{proof}[Proof of Theorem~\ref{thm:complete_split_graphs}]
  Let $\vect y\in[0,1]^{n(n-1)/2}$ be such that $(\vect x,\vect y)\in P$. We need to check that
  $y(E)\geq\sum_{ij\in E}\mu(X_i\cap X_j)$. Let the elements of $S$ be $j_1,j_2,\dots,j_k$
  such that $j_1<j_2<\dots<j_k$. By the third condition on $S$, there is a set $Y\subseteq[0,1)$
  with
  \[\left\lvert\{i\in V_1\cup S\,:\,t\in X_i\}\right\rvert=
    \begin{cases}
      \alpha & \text{for }t\in Y,\\
      \alpha+1 & \text{for }t\in[0,1)\setminus Y.
    \end{cases}
  \]
  Then $x(V_1\cup S)=\alpha+\mu([0,1)\setminus Y)=\alpha+1-\mu(Y)$, and
  \begin{multline*}
    \sum_{i,j\in V_1\cup S,\,i<j}\mu(X_i\cap
    X_j)=\mu(Y)\binom{\alpha}{2}+\left(1-\mu(Y)\right)\binom{\alpha+1}{2}=\binom{\alpha+1}{2}-\mu(Y)\alpha\\
    =\binom{\alpha+1}{2}-(\alpha+1-x(V_1\cup S))\alpha=\alpha x(V_1\cup S)-\binom{\alpha+1}{2}.
  \end{multline*}
  Together with the second condition on $S$, we obtain
  \begin{multline*}
    \sum_{ij\in E}\mu(X_i\cap
    X_j)=\alpha x(V_1\cup S)-\binom{\alpha+1}{2}-x_{j_2}-2x_{j_3}-\dots-(k-1)x_{j_k}\\
    +\sum_{i\in V_1}\sum_{j\in V_2\setminus S}\max\{0,x_i+x_j-1\}.
  \end{multline*}
  Now 
  \[ y(E)=y(E^*(V_1\cup S))-y(E^*(S))+\sum_{i\in V_1}\sum_{j\in V_2\setminus S}y_{ij},\]
  and the required inequality follows from
  \begin{align*}
    y(E^*(V_1\cup S)) &\geq \alpha x(V_1\cup S)-\binom{\alpha+1}{2} &&\text{by~(\ref{eq:clique})},\\
    y(E^*(S)) &\leq x_{j_2}+2x_{j_3}+\dots+(k-1)x_{j_k}
                                                                    &&\text{by~(\ref{eq:McCormick_UB}),\and}\\
    \sum_{i\in V_1}\sum_{j\in V_2\setminus S}y_{ij} &\geq \sum_{i\in V_1}\sum_{j\in V_2\setminus S}\max\{0,x_i+x_j-1\}&&\text{by~(\ref{eq:McCormick_LB}).}\qedhere
  \end{align*}  
\end{proof}
What remains to be done is to construct the set $S$. Algorithm~\ref{alg:extended_construction} does not exactly do that, but it
returns sets $A_p$ from which we can read off the set $S$ easily.
\begin{algorithm}
  \caption{Construction of the sets $X_i$ and $A_p$.}\label{alg:extended_construction}
  \begin{tabbing}
    .....\=.....\=.....\=............................. \kill \\[-1ex]   
    Initialize $(a_0,a_1,\dots,a_{n_2},a_{n_2+1})=(1,x_{n_1+1},x_{n_1+2},\dots,x_{n},0)$\\[1ex]
    Initialize $(A_1,\dots,A_{n_2},A_{n_2+1})=(\{n_1+1\},\{n_1+2\},\dots,\{n\},\{n+1\})$\\[1ex]
    \textbf{for} $i=1,2,\dots,n_1$ \textbf{do}\\[1ex]
    \> $p\leftarrow\min\{k\,:\,a_k+x_i<1\}$\\[1ex]
    \> $X_i\leftarrow[a_{p-1},\,1)\cup[a_p,\,a_p+a_{p-1}+x_i-1)$\\[1ex]
    \> \textbf{if} $p=1$ \textbf{then}\\[1ex]
    \> \> $a_{1}\leftarrow a_{1}+x_i$\\[1ex]
    \> \> $A_1\leftarrow A_1\cup\{i\}$\\[1ex]
    \> \textbf{else}\\[1ex]
    \> \> $a_{p-1}\leftarrow a_p+a_{p-1}+x_i-1$\\[1ex]
    \> \> $A_{p-1}\leftarrow A_p\cup A_{p-1}\cup\{i\}$\\[1ex]
    \> \> \textbf{for} $q=p,p+1,\dots,n_2$ \textbf{do}\\[1ex]
    \> \> \> $a_q\leftarrow a_{q+1}$\\[1ex]
    \> \> \> $A_q\leftarrow A_{q+1}$\\[2ex]
    \textbf{return} $X_1,\dots,X_{n_1}$ and $A_1,\dots,A_{n_2}$
  \end{tabbing}
\end{algorithm}
Let $k$ be the number of $A_p$ with $A_p\cap V_2\neq\emptyset$ when the algorithm terminates, that is,
$0<a_k\leq a_{k-1}\leq\dots\leq a_1<1$. Set $j_p=\max A_p$ for all $p\in[k]$, and let
$j^*=\min(A_{p_0}\cap V_2)$ where $p_0=\min\{p\,:\,\lvert A_p\rvert>1\}$ is the smallest index such
that $A_p$ has been updated by the algorithm. Note that $j_p\in V_2$ for $p<k$, and
$j_k\in\{n,n+1\}$. An important observation for our arguments below is
that $x_{j_p}$ lies between $a_{p+1}$ and $a_p$, and $x_{j^*}$ between $a_{p_0}$ and
$x_{j_{p_0-1}}=a_{p_0-1}$, that is, setting $x_{n+1}=0$ for convenience, we have
\begin{multline*}
  0=a_{k+1}\leq x_{j_k}\leq a_k\leq x_{j_{k-1}}\leq a_{k-1}\leq\dots \leq a_{p_0+1}\leq x_{j_{p_0}}\leq
  a_{p_0}\leq x_{j^*}\leq x_{j_{p_0-1}}=a_{p_0-1}\\
  \leq x_{j_{p_0-2}}= a_{p_0-2}\leq\dots \leq x_{j_2}=a_2\leq x_{j_1}=a_1<1.
\end{multline*}
We are now prepared to specify the set $S$:
\[S=
  \begin{cases}
    V_2\setminus\{j_1,j_2,\dots,j_k,j^*\} & \text{if }p_0\geq 2\text{ or }\lvert A_1\cap
    V_2\rvert=\lvert A_1\cap V_1\rvert+1,\\
    V_2\setminus\{j_1,j_2,\dots,j_k\} & \text{if }p_0=1\text{ and }\lvert A_1\cap
    V_2\rvert\leq \lvert A_1\cap V_1\rvert.
  \end{cases}
\]
Before proving that the set $S$ has the required properties we illustrate the construction in an
example.
\begin{example}\label{ex:split_graphs_2}
  Let $n_1=8$, $n_2=13$, and
\begin{multline*}
  \vect x=(0.9,\,0.85,\,0.53,\,0.49,\,0.44,\,0.23,\,0.16,\,0.1,\\
  0.96,\,0.89,\,0.82,\,0.75,\,0.67,\,0.59,\,0.55,\,0.45,\,0.38,\,0.29,\,0.22,\,0.15,\,0.07)
\end{multline*}
The algorithm proceeds as follows.
\begin{description}
\item[Step 0]
  \begin{align*}
    (a_0,\dots,a_{14}) &=
                         (1,\,0.96,\,0.89,\,0.82,\,0.75,\,0.67,\,0.59,\,0.55,\,0.45,\,0.38,\,0.29,\,0.22,\,0.15,\,0.07,\,0),\\
    (A_1,\dots,A_{14}) &= (\{9\},\,\{10\},\,\{11\},\,\{12\},\,\{13\},\,\{14\},\,\{15\},\,\{16\},\,\{17\},\,\{18\},\,\{19\},\,\{20\},\,\{21\},\,\{22\}).
  \end{align*}
\item[Step 1]
  \begin{align*}
    X_1 &= [0.07,0.12)\cup[0.15,1),\\
    (a_0,\dots,a_{13}) &=
                         (1,\,0.96,\,0.89,\,0.82,\,0.75,\,0.67,\,0.59,\,0.55,\,0.45,\,0.38,\,0.29,\,0.22,\,0.12,\,0),\\
    (A_1,\dots,A_{13}) &= (\{9\},\,\{10\},\,\{11\},\,\{12\},\,\{13\},\,\{14\},\,\{15\},\,\{16\},\,\{17\},\,\{18\},\,\{19\},\,\{1,20,21\},\,\{22\}).
  \end{align*}
\item[Step 2]
  \begin{align*}
    X_2 &= [0.12,0.19)\cup[0.22,1),\\
    (a_0,\dots,a_{12}) &= (1,\,0.96,\,0.89,\,0.82,\,0.75,\,0.67,\,0.59,\,0.55,\,0.45,\,0.38,\,0.29,\,0.19,\,0),\\
    (A_1,\dots,A_{12}) &= (\{9\},\,\{10\},\,\{11\},\,\{12\},\,\{13\},\,\{14\},\,\{15\},\,\{16\},\,\{17\},\,\{18\},\,\{1,2,19,20,21\},\,\{22\}).
  \end{align*}
\item[Step 3]
  \begin{align*}
    X_3 &= [0.45,0.53)\cup[0.55,1),\\
    (a_0,\dots,a_{11}) &= (1,\,0.96,\,0.89,\,0.82,\,0.75,\,0.67,\,0.59,\,0.53,\,0.38,\,0.29,\,0.19,\,0),\\
    (A_1,\dots,A_{11}) &= (\{9\},\,\{10\},\,\{11\},\,\{12\},\,\{13\},\,\{14\},\,\{3,15,16\},\,\{17\},\,\{18\},\,\{1,2,19,20,21\},\,\{22\}).
  \end{align*}
\item[Step 4]
  \begin{align*}
    X_4 &= [0.38,0.4)\cup[0.53,1),\\
    (a_0,\dots,a_{10}) &= (1,\,0.96,\,0.89,\,0.82,\,0.75,\,0.67,\,0.59,\,0.4,\,0.29,\,0.19,\,0),\\
    (A_1,\dots,A_{10}) &= (\{9\},\,\{10\},\,\{11\},\,\{12\},\,\{13\},\,\{14\},\,\{3,4,15,16,17\},\,\{18\},\,\{1,2,19,20,21\},\,\{22\}).
  \end{align*}
\item[Step 5]
  \begin{align*}
    X_5 &= [0.4,0.43)\cup[0.59,1),\\
    (a_0,\dots,a_{9}) &= (1,\,0.96,\,0.89,\,0.82,\,0.75,\,0.67,\,0.43,\,0.29,\,0.19,\,0),\\
    (A_1,\dots,A_{9}) &= (\{9\},\,\{10\},\,\{11\},\,\{12\},\,\{13\},\,\{3,4,5,14,15,16,17\},\,\{18\},\,\{1,2,19,20,21\},\,\{22\}).
  \end{align*}
\item[Step 6]
  \begin{align*}
    X_6 &= [0.75,0.8)\cup[0.82,1),\\
    (a_0,\dots,a_{8}) &= (1,\,0.96,\,0.89,\,0.8,\,0.67,\,0.43,\,0.29,\,0.19,\,0),\\
    (A_1,\dots,A_{8}) &= (\{9\},\,\{10\},\,\{6,11,12\},\,\{13\},\,\{3,4,5,14,15,16,17\},\,\{18\},\,\{1,2,19,20,21\},\,\{22\}).
  \end{align*}      
\item[Step 7]
  \begin{align*}
    X_7 &= [0.8,0.85)\cup[0.89,1),\\
    (a_0,\dots,a_{7}) &= (1,\,0.96,\,0.85,\,0.67,\,0.43,\,0.29,\,0.19,\,0),\\
    (A_1,\dots,A_{7}) &= (\{9\},\,\{6,7,10,11,12\},\,\{13\},\,\{3,4,5,14,15,16,17\},\,\{18\},\,\{1,2,19,20,21\},\,\{22\}).
  \end{align*}
\item[Step 8]
  \begin{align*}
    X_8 &= [0.85,0.91)\cup[0.96,1),\\
    (a_0,\dots,a_{6}) &= (1,\,0.91,\,0.67,\,0.43,\,0.29,\,0.19,\,0),\\
    (A_1,\dots,A_{6}) &= (\{6,7,8,9,10,11,12\},\,\{13\},\,\{3,4,5,14,15,16,17\},\,\{18\},\,\{1,2,19,20,21\},\,\{22\}).
  \end{align*}        
\end{description}
As $p_0=1$, $\lvert A_1\cap V_2\rvert=4=\lvert A_1\cap V_2\rvert+1$, we obtain
$S=\{10,11,14,15,16,19,20\}$. The outcome is illustrated in Figure~\ref{fig:example}.
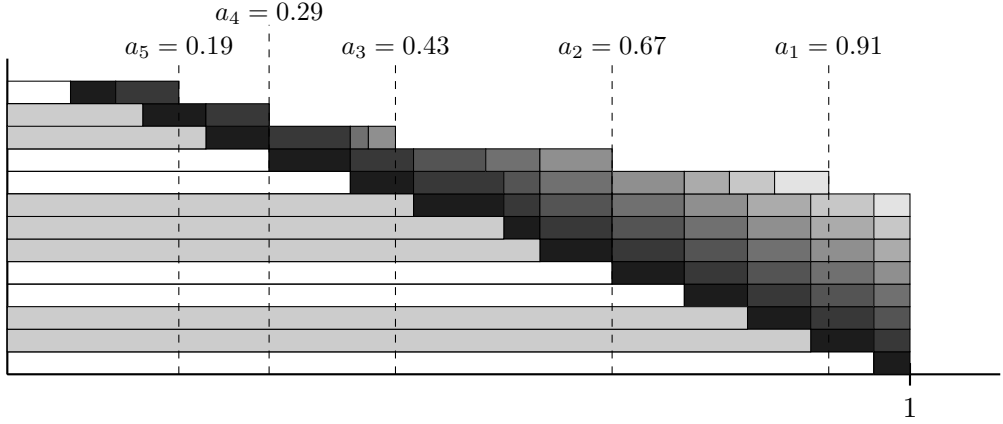
\begin{figure}[htb]
  \centering
  \begin{tikzpicture}[xscale=12,yscale=.3]
\draw[thick] (0,0) -- (1.1,0);
\draw[thick] (0,0) -- (0,14);
\draw[thick] (1,.5) -- (1,-.5);
\node at (1,-1.5) {1};
\draw (0,0) rectangle (24/25,1);
\draw (0,1) rectangle (89/100,2);
\draw (0,2) rectangle (41/50,3);
\draw (0,3) rectangle (3/4,4);
\draw (0,4) rectangle (67/100,5);
\draw (0,5) rectangle (59/100,6);
\draw (0,6) rectangle (11/20,7);
\draw (0,7) rectangle (9/20,8);
\draw (0,8) rectangle (19/50,9);
\draw (0,9) rectangle (29/100,10);
\draw (0,10) rectangle (11/50,11);
\draw (0,11) rectangle (3/20,12);
\draw (0,12) rectangle (7/100,13);
\draw[fill=white!11!black](24/25,0) rectangle (1,1);
\draw[fill=white!11!black](89/100,1) rectangle (24/25,2);
\draw[fill=white!11!black](41/50,2) rectangle (89/100,3);
\draw[fill=white!11!black](3/4,3) rectangle (41/50,4);
\draw[fill=white!11!black](67/100,4) rectangle (3/4,5);
\draw[fill=white!11!black](59/100,5) rectangle (67/100,6);
\draw[fill=white!11!black](11/20,6) rectangle (59/100,7);
\draw[fill=white!11!black](9/20,7) rectangle (11/20,8);
\draw[fill=white!11!black](19/50,8) rectangle (9/20,9);
\draw[fill=white!11!black](29/100,9) rectangle (19/50,10);
\draw[fill=white!11!black](11/50,10) rectangle (29/100,11);
\draw[fill=white!11!black](3/20,11) rectangle (11/50,12);
\draw[fill=white!11!black](7/100,12) rectangle (3/25,13);
\draw[fill=white!22!black](24/25,1) rectangle (1,2);
\draw[fill=white!22!black](89/100,2) rectangle (24/25,3);
\draw[fill=white!22!black](41/50,3) rectangle (89/100,4);
\draw[fill=white!22!black](3/4,4) rectangle (41/50,5);
\draw[fill=white!22!black](67/100,5) rectangle (3/4,6);
\draw[fill=white!22!black](59/100,6) rectangle (67/100,7);
\draw[fill=white!22!black](11/20,7) rectangle (59/100,8);
\draw[fill=white!22!black](9/20,8) rectangle (11/20,9);
\draw[fill=white!22!black](19/50,9) rectangle (9/20,10);
\draw[fill=white!22!black](29/100,10) rectangle (19/50,11);
\draw[fill=white!22!black](11/50,11) rectangle (29/100,12);
\draw[fill=white!22!black](3/25,12) rectangle (19/100,13);
\draw[fill=white!33!black](24/25,2) rectangle (1,3);
\draw[fill=white!33!black](89/100,3) rectangle (24/25,4);
\draw[fill=white!33!black](41/50,4) rectangle (89/100,5);
\draw[fill=white!33!black](3/4,5) rectangle (41/50,6);
\draw[fill=white!33!black](67/100,6) rectangle (3/4,7);
\draw[fill=white!33!black](59/100,7) rectangle (67/100,8);
\draw[fill=white!33!black](11/20,8) rectangle (59/100,9);
\draw[fill=white!33!black](9/20,9) rectangle (53/100,10);
\draw[fill=white!44!black](24/25,3) rectangle (1,4);
\draw[fill=white!44!black](89/100,4) rectangle (24/25,5);
\draw[fill=white!44!black](41/50,5) rectangle (89/100,6);
\draw[fill=white!44!black](3/4,6) rectangle (41/50,7);
\draw[fill=white!44!black](67/100,7) rectangle (3/4,8);
\draw[fill=white!44!black](59/100,8) rectangle (67/100,9);
\draw[fill=white!44!black](53/100,9) rectangle (59/100,10);
\draw[fill=white!44!black](19/50,10) rectangle (2/5,11);
\draw[fill=white!56!black](24/25,4) rectangle (1,5);
\draw[fill=white!56!black](89/100,5) rectangle (24/25,6);
\draw[fill=white!56!black](41/50,6) rectangle (89/100,7);
\draw[fill=white!56!black](3/4,7) rectangle (41/50,8);
\draw[fill=white!56!black](67/100,8) rectangle (3/4,9);
\draw[fill=white!56!black](59/100,9) rectangle (67/100,10);
\draw[fill=white!56!black](2/5,10) rectangle (43/100,11);
\draw[fill=white!67!black](24/25,5) rectangle (1,6);
\draw[fill=white!67!black](89/100,6) rectangle (24/25,7);
\draw[fill=white!67!black](41/50,7) rectangle (89/100,8);
\draw[fill=white!67!black](3/4,8) rectangle (4/5,9);
\draw[fill=white!78!black](24/25,6) rectangle (1,7);
\draw[fill=white!78!black](89/100,7) rectangle (24/25,8);
\draw[fill=white!78!black](4/5,8) rectangle (17/20,9);
\draw[fill=white!89!black](24/25,7) rectangle (1,8);
\draw[fill=white!89!black](17/20,8) rectangle (91/100,9);
\draw[fill=white!80!black](0,1) rectangle (89/100,2);
\draw[fill=white!80!black](0,2) rectangle (41/50,3);
\draw[fill=white!80!black](0,5) rectangle (59/100,6);
\draw[fill=white!80!black](0,6) rectangle (11/20,7);
\draw[fill=white!80!black](0,7) rectangle (9/20,8);
\draw[fill=white!80!black](0,10) rectangle (11/50,11);
\draw[fill=white!80!black](0,11) rectangle (3/20,12);
\draw[dashed] (91/100,0) -- (91/100,14);
\node at (0.91,14.5) {{\small $a_1=0.91$}};
\draw[dashed] (67/100,0) -- (67/100,14);
\node at (0.67,14.5) {{\small $a_2=0.67$}};
\draw[dashed] (43/100,0) -- (43/100,14);
\node at (0.43,14.5) {{\small $a_3=0.43$}};
\draw[dashed] (29/100,0) -- (29/100,15.5);
\node at (0.29,16) {{\small $a_4=0.29$}};
\draw[dashed] (19/100,0) -- (19/100,14);
\node at (0.19,14.5) {{\small $a_5=0.19$}};
\end{tikzpicture}
  \caption{The outcome of Algorithm~\ref{alg:extended_construction} for the vector $\vect x$ in Example~\ref{ex:split_graphs_2}.}
  \label{fig:example}
\end{figure}
\end{example}
The next three lemmas establish the three required properties of the set $S$.
\begin{lemma}\label{lem:size_of_S}
  $\lvert S\rvert\leq n_1-1$.
\end{lemma}
\begin{proof}
 The number of sets $A_p$ with $A_p\cap V_2\neq\emptyset$ is initially equal to $n_2$ and in each iteration it drops by
 1, except when $a_1+x_i<1$ or $a_k+x_i\geq 1$. Therefore, $k\geq n_2-n_1$, and we have the
 following cases:
 \begin{description}
 \item[Case 1] $k=n_2-n_1$. Then $p_0\geq 2$ or $\lvert A_1\cap
    V_2\rvert=\lvert A_1\cap V_1\rvert+1$, hence $\lvert S\rvert=n_2-(k+1)=n_1-1$.
  \item[Case 2] $k=n_2-n_1+1$. Then $j_k=n$ or $p_0\geq 2$ or $\lvert A_1\cap
    V_2\rvert=\lvert A_1\cap V_1\rvert+1$, hence $\lvert S\rvert\leq n_2-k=n_1-1$.
  \item[Case 3] $k\geq n_2-n_1+2$. Then $\lvert S\rvert\leq n_2-(k-1)\leq n_1-1$.\qedhere    
 \end{description}
\end{proof}
\begin{lemma}\label{lem:small_intersections}
  $\mu(X_i\cap X_j)=\max\{0,x_i+x_j-1\}$ for all $i\in V_1$, $j\in V_2\setminus S$.
\end{lemma}
\begin{proof}
  Fix $j\in V_2\setminus S$, $i\in V_1$. Then $j=j_p$ for some $p$ or $j=j^*$, and $i\in A_q$ for
  some $q$. First assume $j=j_p$. If $q\leq p$ then $X_i\subseteq[x_j,1)$, hence
  $X_i\cap X_j=\emptyset$, and if $q>p$, then $X_i\supseteq[x_j,1)$, hence $\mu(X_i\cap
  X_j)=x_i+x_j-1$. Now assume $j=j^*$. If $q<p_0$ then $X_i\subseteq[x_j,1)$, hence
  $X_i\cap X_j=\emptyset$, and if $q\geq p_0$, then $X_i\supseteq[x_j,1)$, hence $\mu(X_i\cap
  X_j)=x_i+x_j-1$.
\end{proof}
Define a function $h:[0,1)\to\ints$ by $h(t)=\lvert\{i\in V_1\cup S\,:\,t\in X_i\}\rvert$. Our third
condition on $S$ says that the function $h$ takes only two consecutive values.
\begin{example}\label{ex:function_h}
  The graph of the function $h$ for the instance in Example~\ref{ex:split_graphs_2} is shown in Figure~\ref{fig:h}.
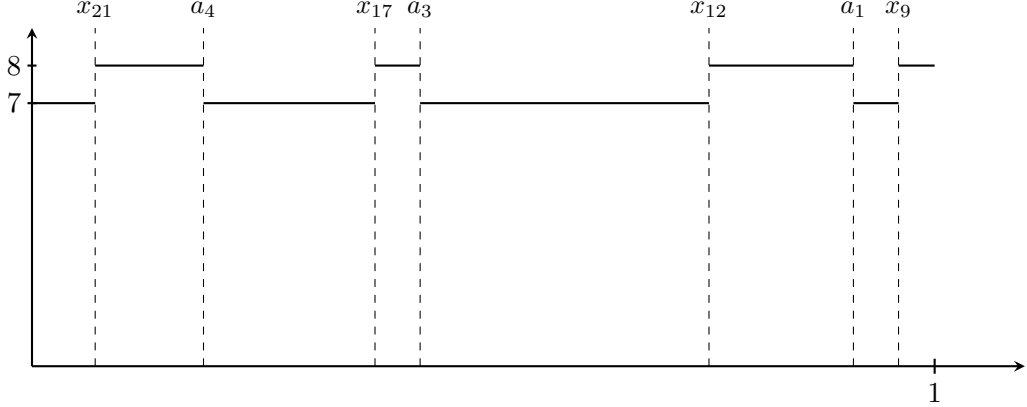
\begin{figure}[htb]
  \centering
  \begin{tikzpicture}[xscale=12,yscale=.5]
    \draw[thick,->] (0,0) -- (1.1,0);
    \draw[thick,->] (0,0) -- (0,9);
    \draw[thick] (1,.2) -- (1,-.2);
    \draw[thick] (-.005,7) -- (.005,7);
    \draw[thick] (-.005,8) -- (.005,8);
    \node at (1,-.7) {1};
    \node at (-.02,7) {7};
    \node at (-.02,8) {8};
    \draw[thick] (0,7) -- (0.07,7);
    \draw[thick] (0.07,8) -- (0.19,8);
    \draw[thick] (0.19,7) -- (0.38,7);
    \draw[thick] (0.38,8) -- (0.43,8);
    \draw[thick] (0.43,7) -- (0.75,7);
    \draw[thick] (0.75,8) -- (0.91,8);
    \draw[thick] (0.91,7) -- (0.96,7);
    \draw[thick] (0.96,8) -- (1,8);
    \draw[dashed] (0.07,0) -- (0.07,9); \node at (0.07,9.5) {{\small $x_{21}$}};
    \draw[dashed] (0.19,0) -- (0.19,9); \node at (0.19,9.5) {{\small $a_4$}};
    \draw[dashed] (0.38,0) -- (0.38,9); \node at (0.38,9.5) {{\small $x_{17}$}};
    \draw[dashed] (0.43,0) -- (0.43,9); \node at (0.43,9.5) {{\small $a_3$}};
    \draw[dashed] (0.75,0) -- (0.75,9); \node at (0.75,9.5) {{\small $x_{12}$}};
    \draw[dashed] (0.91,0) -- (0.91,9); \node at (0.91,9.5) {{\small $a_{1}$}};
    \draw[dashed] (0.96,0) -- (0.96,9); \node at (0.96,9.5) {{\small $x_{9}$}};
  \end{tikzpicture}
  \caption{The function $h$ for Example~\ref{ex:split_graphs_2}.}
  \label{fig:h}
\end{figure}
\end{example}

\begin{lemma}\label{lem:balancedness}
  There exists an integer $\alpha\in\{0,1,\dots,n_1-1\}$ such that $h(t)\in\{\alpha,\alpha+1\}$ for
  all $t\in[0,1)$.  
\end{lemma}
\begin{proof}
  We start with the function $h_0$ defined by $h_0(t)=\lvert\{i\in V_1\cup V_2\,:\,t\in
  X_i\}\rvert$. Denoting the value of $h_0$ on the interval $[a_1,1)$ by $\gamma$, we have
  $h_0(t)=\gamma+p-1$ for $t\in[a_{p},a_{p-1})$, $p=1,2,\dots,k+1$.
  \begin{description}
  \item[Case 1] $S=V_2\setminus\{j_1,j_2,\dots,j_k,j^*\}$. Then
    \[h(t)=
      \begin{cases}
        h_0(t)-(p-1)   &\text{for }t\in[a_p,a_{p-1}),\,p=1,2,\dots,p_0-1,\\
        h_0(t)-(p_0-1) &\text{for }t\in[x_{j^*},a_{p_0-1}),\\
        h_0(t)-p_0 &\text{for }t\in[a_{p_0},\,x_{j^*}),\\
        h_0(t)-(p-1) &\text{for }t\in[x_{j_{p-1}},\,a_{p-1}),\,p=p_0+1,\dots,k,\\
        h_0(t)-p &\text{for }t\in[a_p,x_{j_{p-1}}),\,p=p_0+1,\dots,k+1.
      \end{cases}      
    \]
    Substituting $h_0(t)=\gamma+p-1$ for $t\in[a_{p},a_{p-1})$, $p=1,2,\dots,k+1$, we obtain
    \[h(t)=
      \begin{cases}
        \gamma & \text{for }t\in[x_{j^*},1)\\
        \gamma-1 & \text{for }t\in[a_{p_0},x_{j^*}),\\
        \gamma &\text{for }t\in[x_{j_p},a_p),\,p=p_0,\dots,k,\\
        \gamma-1 &\text{for }t\in[a_{p+1},x_{j_p}),\,p=p_0,\dots,k,
      \end{cases}
    \]
    and this concludes the proof with $\alpha=\gamma-1$.
  \item[Case 2] $S=V_2\setminus\{j_1,j_2,\dots,j_k\}$. Then $p_0=1$ and $h(t)=h_0(t)-(p-1)$ for
    $t\in[x_{j_p},x_{j_{p-1}})$, $p=1,2,\dots,k+1$ (with $x_{j_{k+1}}=0$ and
    $x_{j_0}=1$). Substituting $h_0(t)=\gamma+p-1$ for $t\in[a_{p},a_{p-1})$, $p=1,2,\dots,k+1$, we
    obtain
    \[h(t)=
      \begin{cases}
        \gamma & \text{for }t\in[a_p,x_{j_{p-1}}),\,p=1,\dots,k,\\
        \gamma+1 & \text{for }t\in[x_{j_p},a_p),\,p=1,\dots,k+1,
      \end{cases}
    \]
    and this concludes the proof with $\alpha=\gamma$.\qedhere    
  \end{description}

\end{proof}

\section{Conclusion}\label{sec:conclusion}
In this paper, we have extended the results from \cite{Gupte2020} to derive extended formulations
for the graphs of quadratic functions corresponding to even wheel graphs and complete split
graphs. It is a natural question what happens for odd wheels. For odd wheels, the triangle
inequalities are not sufficient. In fact, even adding in all the facets from~\cite{Padberg1989} is
not enough. In Figure~\ref{fig:5-wheel} this is illustrated for the $5$-wheel. For the two points
shown in the picture we have $\pi[f](\vect x,\vect y)=(1/3,\,1/3,\,1/3,\,1/3,\,1/3,\,2/3,\,5/6)$ and
$\pi[f](\vect x,\vect y)=(2/3,\,2/3,\,2/3,\,2/3,\,2/3,\,1/3,\,5/2)$, respectively, hence
$\pi[f](P)\neq X(f)$, because $(\vect x,\vect y)\in\QP_6(W_5)$ implies
\begin{align}
  y(E) &\geq 2x_6+x_1+x_2+x_3+x_4+x_5-2,\label{eq:5-wheel_1}\\
  y(E) &\geq 3x_6+2(x_1+x_2+x_3+x_4+x_5)-5.\label{eq:5-wheel_2}
\end{align}

\def\ca{0.61803398875}
\def\cb{1.61803398875}
\def\sa{1.90211303259}
\def\sb{1.17557050458}
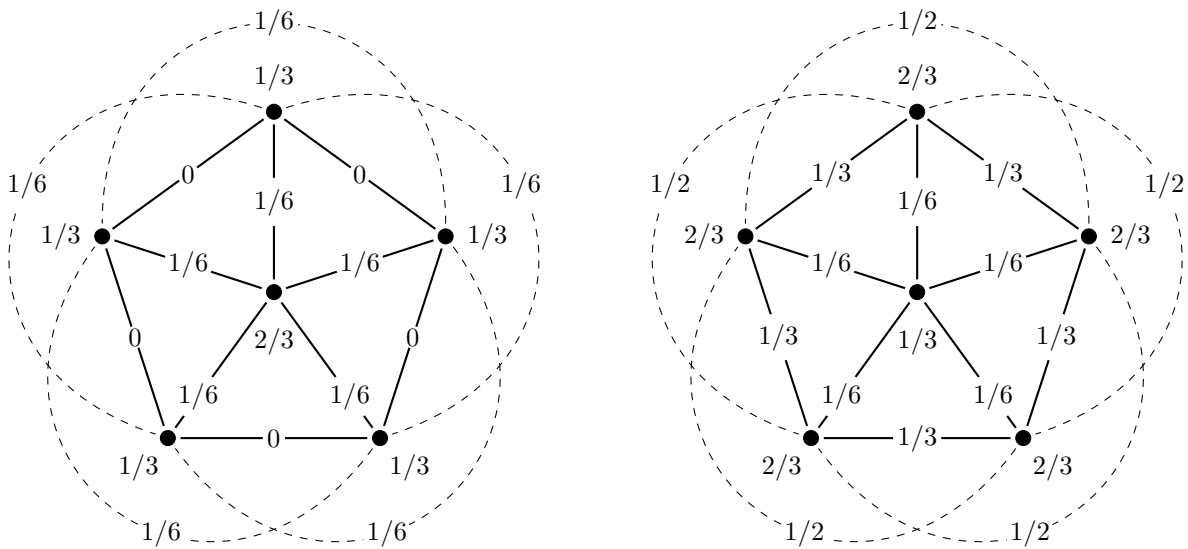
\begin{figure}[htb]
  \centering
  \begin{minipage}{.49\linewidth}
    \centering
    \begin{tikzpicture}[scale=1.2,every node/.style={circle,fill=black,draw,inner sep=2pt,outer sep=2pt}]
\tikzset{n/.style={circle,draw,inner sep=2pt,outer sep=2pt}}
\tikzset{l/.style={draw=none,fill=white,inner sep=.1pt,outer sep=.1pt}}

\clip (0,0) circle (3.5);

\node[label={[label distance=0mm]below:{\small $2/3$}}] (0) at (0,0) {};
\node[label={[label distance=-2mm]above:{\small $1/3$}}] (1) at (0,2) {};
\node[label={[label distance=-1mm]right:{\small $1/3$}}] (2) at (\sa,\ca) {};
\node[label={[label distance=-1mm]below right:{\small $1/3$}}] (3) at (\sb,-\cb) {};
\node[label={[label distance=-1mm]below left:{\small $1/3$}}] (4) at (-\sb,-\cb) {};
\node[label={[label distance=-1mm]left:{\small $1/3$}}] (5) at (-\sa,\ca) {};

\draw[thick] (0) -- (1) node[midway,l] {{\small $1/6$}};
\draw[thick] (0) -- (2) node[midway,l] {{\small $1/6$}};
\draw[thick] (0) -- (3) node[near end,l] {{\small $1/6$}};
\draw[thick] (0) -- (4) node[near end,l] {{\small $1/6$}};
\draw[thick] (0) -- (5) node[midway,l] {{\small $1/6$}};

\draw[thick] (1) -- (2) node[midway,l] {{\small $0$}};
\draw[thick] (2) -- (3) node[midway,l] {{\small $0$}};
\draw[thick] (3) -- (4) node[midway,l] {{\small $0$}};
\draw[thick] (4) -- (5) node[midway,l] {{\small $0$}};
\draw[thick] (1) -- (5) node[midway,l] {{\small $0$}};

\path[dashed] (1) edge[bend left=90,looseness=2] node[pos=0.45,l] {{\small $1/6$}} (3);
\path[dashed] (1) edge[bend right=90,looseness=2] node[pos=0.45,l] {{\small $1/6$}} (4);
\path[dashed] (2) edge[bend left=90,looseness=2] node[pos=0.6,l] {{\small $1/6$}} (4);
\path[dashed] (2) edge[bend right=90,looseness=2] node[midway,l] {{\small $1/6$}} (5);
\path[dashed] (3) edge[bend left=90,looseness=2] node[pos=0.4,l] {{\small $1/6$}} (5);
\end{tikzpicture}
\end{minipage}\hfill
\begin{minipage}{.49\linewidth}
  \centering
  \begin{tikzpicture}[scale=1.2,every node/.style={circle,fill=black,draw,inner sep=2pt,outer sep=2pt}]
\tikzset{n/.style={circle,draw,inner sep=2pt,outer sep=2pt}}
\tikzset{l/.style={draw=none,fill=white,inner sep=.1pt,outer sep=.1pt}}

\clip (0,0) circle (3.5);

\node[label={[label distance=0mm]below:{\small $1/3$}}] (0) at (0,0) {};
\node[label={[label distance=-2mm]above:{\small $2/3$}}] (1) at (0,2) {};
\node[label={[label distance=-1mm]right:{\small $2/3$}}] (2) at (\sa,\ca) {};
\node[label={[label distance=-1mm]below right:{\small $2/3$}}] (3) at (\sb,-\cb) {};
\node[label={[label distance=-1mm]below left:{\small $2/3$}}] (4) at (-\sb,-\cb) {};
\node[label={[label distance=-1mm]left:{\small $2/3$}}] (5) at (-\sa,\ca) {};

\draw[thick] (0) -- (1) node[midway,l] {{\small $1/6$}};
\draw[thick] (0) -- (2) node[midway,l] {{\small $1/6$}};
\draw[thick] (0) -- (3) node[near end,l] {{\small $1/6$}};
\draw[thick] (0) -- (4) node[near end,l] {{\small $1/6$}};
\draw[thick] (0) -- (5) node[midway,l] {{\small $1/6$}};

\draw[thick] (1) -- (2) node[midway,l] {{\small $1/3$}};
\draw[thick] (2) -- (3) node[midway,l] {{\small $1/3$}};
\draw[thick] (3) -- (4) node[midway,l] {{\small $1/3$}};
\draw[thick] (4) -- (5) node[midway,l] {{\small $1/3$}};
\draw[thick] (1) -- (5) node[midway,l] {{\small $1/3$}};

\path[dashed] (1) edge[bend left=90,looseness=2] node[pos=0.45,l] {{\small $1/2$}} (3);
\path[dashed] (1) edge[bend right=90,looseness=2] node[pos=0.45,l] {{\small $1/2$}} (4);
\path[dashed] (2) edge[bend left=90,looseness=2] node[pos=0.6,l] {{\small $1/2$}} (4);
\path[dashed] (2) edge[bend right=90,looseness=2] node[midway,l] {{\small $1/2$}} (5);
\path[dashed] (3) edge[bend left=90,looseness=2] node[pos=0.4,l] {{\small $1/2$}} (5);
\end{tikzpicture}
\end{minipage}
\caption{For the 5-wheel the relaxation of $\BQP_6$ obtained by using all the facets
  from~\cite{Padberg1989} is not sufficient to describe $X(f)$. The vector $(\vect x,\vect y)$
  corresponding to the left picture satisfies all these inequalities and~(\ref{eq:5-wheel_2}), but
  not~(\ref{eq:5-wheel_1}), while vector corresponding to the right picture satisfies the inequalities
  from~\cite{Padberg1989} and~(\ref{eq:5-wheel_1}), but not~(\ref{eq:5-wheel_2}).}
  \label{fig:5-wheel}
\end{figure}

For the 5-wheel it can be checked using polymake~\cite{Assarf2017}, that $\pi[f](P)=X(f)$
where $P$ is the polytope described by the McCormick inequalities, the inequalities
$y_{ij}+y_{jk}+y_{ik}\geq x_i+x_j+x_k-1$ for all triangles in $W_5$ and the
inequalities~(\ref{eq:5-wheel_1}) and~(\ref{eq:5-wheel_2}). This corrects the observation about the
5-wheel in\cite{Gupte2020}, where the triangle inequalities were missing.
\begin{problem}\label{prob:odd-wheel}
  Let $n-1$ be odd, let $f:[0,1]^n\to\reals$ be the function corresponding to the wheel graph
  $G=W_{n-1}$, and let $P\subseteq[0,1]^{3n-2}$ be the polytope described by the
  McCormick inequalities, the inequalities $y_{ij}+y_{jk}+y_{ik}\geq x_i+x_j+x_k-1$ for all
  triangles in $G$, and the two inequalities
  \begin{align*}
  y(E) &\geq \frac{n-2}{2}x_n+\sum_{i=1}^{n-1}x_i-\frac{n-2}{2}, & y(E) &\geq \frac{n}{2}x_n+2\sum_{i=1}^{n-1}x_i-(n-1).
  \end{align*}
  Is it true that $X(f)=\pi[f](P)$?
\end{problem}
As mentioned in the introduction, it would be interesting to characterize the graphs $G$ for which the
polytope $T(G)$ (see Definition~\ref{def:triangle_relaxation}) is sufficient. 
\begin{problem}\label{prob:characterisation}
  Characterise the graphs $G$ such that $\pi[f](T(G))=X(f)$.
\end{problem}
As a first step, one can try to generalise Theorem~\ref{thm:even_wheels} in various directions.
\begin{problem}\label{prob:ext_wheel}
  Let $n=n_1+n_2$ with $n_1$ even, and let $G$ be the graph obtained by joining every vertex of a
  cycle of length $n_1$ to all vertices of an independent set of size $n_2$. Is it true that $\pi[f](T(G))=X(f)$?
\end{problem}
\begin{problem}\label{prob:bipartite_plus_universal}
Let $G$ be a graph on $n$ vertices, where the graph induced on $\{1,\dots,n-1\}$ is bipartite, and
vertex $n$ is adjacent to all vertices in $\{1,\dots,n-1\}$. Is it true that $\pi[f](T(G))=X(f)$?  
\end{problem}
Theorem~\ref{thm:even_wheels} is the special case for Problem~\ref{prob:ext_wheel} with $n_2=1$, and
for Problem~\ref{prob:bipartite_plus_universal} with the induced bipartite graph being a single even
cycle. Another class of graphs for which it seems plausible that $\pi[f](T(G))=X(f)$ are the
$2$-trees, that is, the graphs obtained by starting with a triangle, and repeatedly adding a vertex
and joining to two adjacent vertices (maximal series-parallel graphs).
\begin{problem}
 Is it true that $\pi[f](T(G))=X(f)$ whenever $G$ is a 2-tree?
\end{problem}

\printbibliography

\begin{appendix}
  \section{Proof of Lemma~\ref{lem:no_negative_cycles}}\label{app:proof_details}
  In this appendix we provide the details for the proof of
  Lemma~\ref{lem:no_negative_cycles}. Recall that we have fixed a set $T^*\subseteq[n-1]$ satisfying
  the following conditions:
  \begin{enumerate}
\item $T^*\cap(T^*+1)=\emptyset$, and 
\item $\phi(T^*)=\Phi^*$, and
\item $1\leq x_i+x_{i+1}+x_n\leq 2$ for all $i\in T$, and
\item $T^*\cap\{i-1,i,i+1\}\neq\emptyset$ for all $i\in[n-1]$ with $1\leq x_i+x_{i+1}+x_n\leq 2$.
\end{enumerate}
We have then defined a network $N=N(\vect x,T^*)$, and we have to verify that the network does not
contain negative cost cycles. Let $C$ be a negative cost cycle with the minimum number of arcs. We
discuss the cases indicated in the proof outline separately.  
\subsection{Case 1} $C=(1,2,\dots,n-1)$. By the definition of the arc set of $N$, all elements of
$T^*$ are odd because the arc $(i,i+1)$ does not exist for even $i\in T^*$. The last two
properties of $T^*$ listed above imply
\[T^*=\left\{i\in[n-1]\,:\,i\text{ is odd and }1\leq x_i+x_{i+1}+x_n\leq 2\right\}.\]
We set $T=\left\{i\in[n-1]\,:\,i\text{ is even and }1\leq x_i+x_{i+1}+x_n\leq 2\right\}$, and want
to argue that $\phi(T)>\phi(T^*)$. We start by writing down the cost of $C$ in a convenient form:
\begin{multline*}
  \cost(C)=-\sum_{i\in T^*}M'_i-\sum_{i\in[n-1]\setminus T^*\text{ odd}}m'_i+\sum_{i\in[n-1]\text{
      even}}M'_i \\
  \sum_{i\in T^*}\left(x_n-\max\{1,x_i+x_{i+1}\}\right)+\sum_{\text{odd }i\in[n-1]\setminus T^*}\left(x_n-\min\{1,x_i+x_{i+1}\}\right)+\sum_{\text{even }i\in[n-1]}\left(\max\{1,x_i+x_{i+1}\}-x_n\right).
\end{multline*}
We observe that the $x_n$ cancel, and we split the second sum into the sum over $i\in S_0$ and $i\in
S_2$, where
\begin{align*}
S_0 &= \{i\in[n-1]\,:\,i\text{ odd and }x_i+x_{i+1}+x_n<1\}, & S_2 &= \{i\in[n-1]\,:\,i\text{ odd and }x_i+x_{i+1}+x_n>2\}.   
\end{align*}
In addition, we use
\[\min\{1,x_i+x_{i+1}\}=\max\{1,x_i+x_{i+1}\}-
  \begin{cases}
    1-x_i-x_{i+1} & \text{for }i\in S_0,\\
    x_i+x_{i+1}-1 & \text{for }i\in S_2.
  \end{cases}
\]
Putting this together, we obtain
\begin{multline}\label{eq:cycle_cost}
  \cost(C)=\sum_{\text{even }i\in[n-1]}\max\{1,x_i+x_{i+1}\}-\sum_{\text{odd
    }i\in[n-1]}\max\{1,x_i+x_{i+1}\}\\
  +\sum_{i\in S_0}(1-x_i-x_{i+1})+\sum_{i\in S_2}(x_i+x_{i+1}-1).
\end{multline}
Now let's write down $\phi(T^*)$ and $\phi(T)$, and then $\phi(T)-\phi(T^*)$. We are done once we
are convinced that $\phi(T)-\phi(T^*)>0$. The contributions to $\phi(T^*)$ are as follows:
\begin{itemize}
\item An $i\in T^*$ contributes $x_i+x_{i+1}+x_n-1$.
\item An $i\in S_0$ contributes $0=x_i+x_{i+1}+x_n-1+(1-x_i-x_{i+1}-x_n)$.
\item An $i\in S_2$ contributes $2x_i+2x_{i+1}+2x_n-3=x_i+x_{i+1}+x_n-1+(x_i+x_{i+1}+x_n-2)$.
\item An even $i\in[n-1]$ contributes $\max\{0,x_i+x_i-1\}=\max\{1,x_i+x_{i+1}\}-1$.
\end{itemize}
We obtain
\begin{multline*}
  \phi(T^*)=\sum_{\text{odd }i\in[n-1]}(x_i+x_{i+1}+x_n-1)+\sum_{\text{even
    }i\in[n-1]}\left(\max\{1,x_i+x_{i+1}\}-1\right)\\
  +\sum_{i\in S_0}(1-x_i-x_{i+1}-x_n)+\sum_{i\in S_2}(x_i+x_{i+1}+x_n-2),
\end{multline*}
which simplifies to
\begin{multline*}
  \phi(T^*)=\sum_{i\in[n-1]}x_i+\frac{n-1}{2}(x_n-2)+\sum_{\text{even
    }i\in[n-1]}\max\{1,x_i+x_{i+1}\}\\
  +\sum_{i\in S_0}(1-x_i-x_{i+1}-x_n)+\sum_{i\in S_2}(x_i+x_{i+1}+x_n-2).
\end{multline*}
Similarly,
\begin{multline*}
  \phi(T)=\sum_{i\in[n-1]}x_i+\frac{n-1}{2}(x_n-2)+\sum_{\text{odd
    }i\in[n-1]}\max\{1,x_i+x_{i+1}\}\\
  +\sum_{i\in S'_0}(1-x_i-x_{i+1}-x_n)+\sum_{i\in S'_2}(x_i+x_{i+1}+x_n-2),
\end{multline*}
where
\begin{align*}
S'_0 &= \{i\in[n-1]\,:\,i\text{ even and }x_i+x_{i+1}+x_n<1\}, & S_2 &= \{i\in[n-1]\,:\,i\text{ even and }x_i+x_{i+1}+x_n>2\}.   
\end{align*}
As a consequence,
\begin{multline*}
  \phi(T)-\phi(T^*)=\sum_{\text{odd
    }i\in[n-1]}\max\{1,x_i+x_{i+1}\}-\sum_{\text{even
    }i\in[n-1]}\max\{1,x_i+x_{i+1}\}\\ +\sum_{i\in S'_0}(1-x_i-x_{i+1}-x_n)+\sum_{i\in S'_2}(x_i+x_{i+1}+x_n-2)-\sum_{i\in S_0}(1-x_i-x_{i+1}-x_n)-\sum_{i\in S_2}(x_i+x_{i+1}+x_n-2).
\end{multline*}
Using~(\ref{eq:cycle_cost}) this can be rewritten as
\[\phi(T)-\phi(T^*)=-\cost(C)+\sum_{i\in S'_0}(1-x_i-x_{i+1}-x_n)+\sum_{i\in
    S'_2}(x_i+x_{i+1}+x_n-2)+\lvert S_0\rvert x_n+\lvert S_2\rvert(1-x_n).\]
This implies the required inequality $\phi(T)-\phi(T^*)>0$ because $-\cost(C)>0$ and the remaining
terms on the right hand side are non-negative.
  \subsection{Case 2} $C=(O,i,i+1,\dots,j,j+1,O)$. The assumption that $C$ is a shortest cycle of
  negative cost leads to some restrictions on $i$ and $j$ as it implies
  $\cost(O,i)+\cost(i,i+1)<\cost(O,i+1)$ and $\cost(j,j+1)+\cost(j+1,O)<\cost(j,O)$.
  \begin{lemma}\label{lem:case_3_possible_starts_1}
    $\{i,i+1\}\cap T^*\neq\emptyset$.
  \end{lemma}
  \begin{proof}
    Suppose $\{i,i+1\}\cap T^*=\emptyset$. We distinguish the three cases shown in Figure~\ref{fig:case_3_possible_starts_1}.
    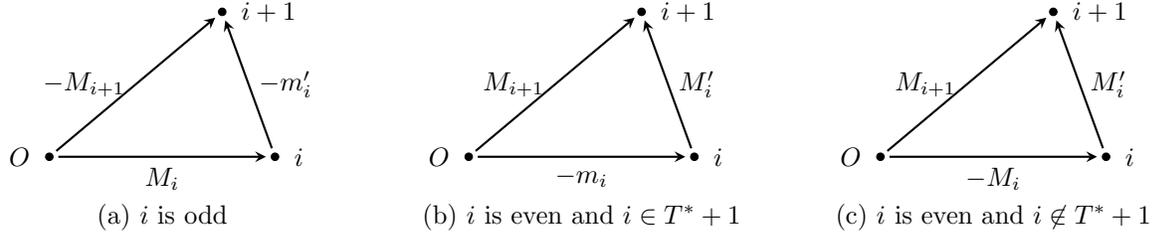
\begin{figure}[htb]
      \begin{minipage}{.32\linewidth}
        \begin{tikzpicture}
          \node[draw,fill,circle,outer sep=2pt,inner sep=1pt,label={left:{\small $O$}}] (O) at (0,0) {};
          \node[draw,fill,circle,outer sep=2pt,inner sep=1pt,label={right:{\small $i$}}]
          (i) at (0:3) {};
          \node[draw,fill,circle,outer sep=2pt,inner sep=1pt,label={right:{\small $i+1$}}]
          (j) at (40:3) {};
          \draw[thick,->] (O) to node[below] {{\small $M_i$}} (i);
          \draw[thick,->] (i) to node[right] {{\small $-m'_i$}} (j);
          \draw[thick,->] (O) to node[left] {{\small $-M_{i+1}$}} (j);
          \node at (1.5,-.8) {{\small (a) $i$ is odd}};
        \end{tikzpicture}
      \end{minipage}
      \begin{minipage}{.32\linewidth}
        \begin{tikzpicture}
          \node[draw,fill,circle,outer sep=2pt,inner sep=1pt,label={left:{\small $O$}}] (O) at (0,0) {};
          \node[draw,fill,circle,outer sep=2pt,inner sep=1pt,label={right:{\small $i$}}]
          (i) at (0:3) {};
          \node[draw,fill,circle,outer sep=2pt,inner sep=1pt,label={right:{\small $i+1$}}]
          (j) at (40:3) {};
          \draw[thick,->] (O) to node[below] {{\small $-m_i$}} (i);
          \draw[thick,->] (i) to node[right] {{\small $M'_i$}} (j);
          \draw[thick,->] (O) to node[left] {{\small $M_{i+1}$}} (j);
          \node at (1.5,-.8) {{\small (b) $i$ is even and $i\in T^*+1$}};
        \end{tikzpicture}
      \end{minipage}
      \begin{minipage}{.32\linewidth}
        \begin{tikzpicture}
          \node[draw,fill,circle,outer sep=2pt,inner sep=1pt,label={left:{\small $O$}}] (O) at (0,0) {};
          \node[draw,fill,circle,outer sep=2pt,inner sep=1pt,label={right:{\small $i$}}]
          (i) at (0:3) {};
          \node[draw,fill,circle,outer sep=2pt,inner sep=1pt,label={right:{\small $i+1$}}]
          (j) at (40:3) {};
          \draw[thick,->] (O) to node[below] {{\small $-M_i$}} (i);
          \draw[thick,->] (i) to node[right] {{\small $M'_i$}} (j);
          \draw[thick,->] (O) to node[left] {{\small $M_{i+1}$}} (j);
          \node at (1.5,-.8) {{\small (c) $i$ is even and $i\not\in T^*+1$}};
        \end{tikzpicture}
      \end{minipage}
      \caption{The three cases in the proof of Lemma~\ref{lem:case_3_possible_starts_1}.}
      \label{fig:case_3_possible_starts_1}
    \end{figure}
    \begin{description}
    \item[Case (a)] $i$ is odd. We need to show that
      \[\min\{x_i,1-x_n\}-\min\{1,x_i+x_{i+1}\}+x_n\geq-\min\{x_{i+1},1-x_n\}=\max\{-x_{i+1},x_n-1\}.\]
      \begin{itemize}
      \item If $x_i+x_n\leq 1$ then the left hand side is $x_i-\min\{1,x_i+x_{i+1}\}+x_n\geq
        x_n-x_{i+1}$.
      \item If $x_i+x_n> 1$ then the left hand side is $1-\min\{1,x_i+x_{i+1}\}\geq 0$.
      \end{itemize}
    \item[Case (b)] $i$ is even and $i\in T^*+1$. We need to show that
      \[-\max\{0,x_i-x_n\}+\max\{1,x_i+x_{i+1}\}-x_n\geq\min\{x_{i+1},1-x_n\}.\]
      \begin{itemize}
      \item If $x_i\leq x_n$ then the left hand side is
        $\max\{1,x_i+x_{i+1}\}-x_n\geq 1-x_n$.
      \item If $x_i>x_n$ then the left hand side is $\max\{1,x_i+x_{i+1}\}-x_i\geq x_{i+1}$.
      \end{itemize}
    \item[Case (c)] $i$ is even and $i\not\in T^*+1$. We need to show that
      \[-\min\{x_i,1-x_n\}+\max\{1,x_i+x_{i+1}\}-x_n\geq\min\{x_{i+1},1-x_n\}.\]
      From $T^*\cap\{i-1,i,i+1\}=\emptyset$ it follows that $x_i+x_{i+1}+x_n<1$ or $x_i+x_{i+1}+x_n>2$.
      \begin{itemize}
      \item If $x_i+x_{i+1}+x_n<1$ then the left hand side is $1-x_i-x_n\geq x_{i+1}$.
      \item If $x_i+x_{i+1}+x_m>2$ then the left hand side is $x_i+x_{i+1}-1\geq 1-x_n$.\qedhere
      \end{itemize}
    \end{description}
  \end{proof}
  \begin{lemma}\label{lem:case_3_possible_starts_2}\hfill
    \begin{enumerate}
    \item If $i\in T^*$ then $i$ is odd and either $x_i+\max\{x_{i+1},x_n\}<1$ or
      $x_i+\min\{x_{i+1},x_n\}>1$.
    \item If $i+1\in T^*$ then $i$ is even and $i-1\not\in T^*$. 
    \end{enumerate}
  \end{lemma}
  \begin{proof}
    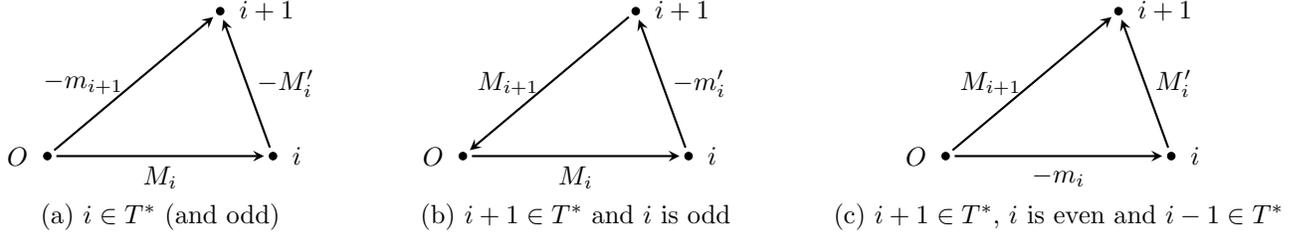
\begin{figure}[htb]
      \centering
      \begin{minipage}{.32\linewidth}
        \begin{tikzpicture}
          \node[draw,fill,circle,outer sep=2pt,inner sep=1pt,label={left:{\small $O$}}] (O) at (0,0) {};
          \node[draw,fill,circle,outer sep=2pt,inner sep=1pt,label={right:{\small $i$}}]
          (i) at (0:3) {};
          \node[draw,fill,circle,outer sep=2pt,inner sep=1pt,label={right:{\small $i+1$}}]
          (j) at (40:3) {};
          \draw[thick,->] (O) to node[below] {{\small $M_i$}} (i);
          \draw[thick,->] (i) to node[right] {{\small $-M'_i$}} (j);
          \draw[thick,->] (O) to node[left] {{\small $-m_{i+1}$}} (j);
          \node at (1.5,-.8) {{\small (a) $i\in T^*$ (and odd)}};
        \end{tikzpicture}
      \end{minipage}
      \begin{minipage}{.32\linewidth}
        \begin{tikzpicture}
          \node[draw,fill,circle,outer sep=2pt,inner sep=1pt,label={left:{\small $O$}}] (O) at (0,0) {};
          \node[draw,fill,circle,outer sep=2pt,inner sep=1pt,label={right:{\small $i$}}]
          (i) at (0:3) {};
          \node[draw,fill,circle,outer sep=2pt,inner sep=1pt,label={right:{\small $i+1$}}]
          (j) at (40:3) {};
          \draw[thick,->] (O) to node[below] {{\small $M_i$}} (i);
          \draw[thick,->] (i) to node[right] {{\small $-m'_i$}} (j);
          \draw[thick,<-] (O) to node[left] {{\small $M_{i+1}$}} (j);
          \node at (1.5,-.8) {{\small (b) $i+1\in T^*$ and $i$ is odd}};
        \end{tikzpicture}
      \end{minipage}
      \begin{minipage}{.32\linewidth}
        \begin{tikzpicture}
          \node[draw,fill,circle,outer sep=2pt,inner sep=1pt,label={left:{\small $O$}}] (O) at (0,0) {};
          \node[draw,fill,circle,outer sep=2pt,inner sep=1pt,label={right:{\small $i$}}]
          (i) at (0:3) {};
          \node[draw,fill,circle,outer sep=2pt,inner sep=1pt,label={right:{\small $i+1$}}]
          (j) at (40:3) {};
          \draw[thick,->] (O) to node[below] {{\small $-m_i$}} (i);
          \draw[thick,->] (i) to node[right] {{\small $M'_i$}} (j);
          \draw[thick,->] (O) to node[left] {{\small $M_{i+1}$}} (j);
          \node at (1.5,-.8) {{\small (c) $i+1\in T^*$, $i$ is even and $i-1\in T^*$}};
        \end{tikzpicture}
      \end{minipage}
      \caption{Illustrations for the proof of Lemma~\ref{lem:case_3_possible_starts_2}}
      \label{fig:case_3_possible_starts_2}
    \end{figure}
    \begin{enumerate}
    \item If $i\in  T^*$ then $i$ has to be odd, since otherwise the arc $(i,i+1)$ does not exist.
      Assume that
    $x_i+\max\{x_{i+1},x_n\}\geq 1$ and $x_i+\min\{x_{i+1},x_n\}\leq 1$. We have to show that
    \[\min\{x_i,1-x_n\}-\max\{1,x_{i}+x_{i+1}\}+x_n\geq
      -\max\{0,x_{i+1}-x_n\}=\min\{0,x_n-x_{i+1}\}.\]
    \begin{itemize}
    \item If $x_i+x_{i+1}\geq 1$ and $x_i+x_n\leq 1$, then the left hand side is $x_n-x_{i+1}$.
    \item If $x_i+x_{i+1}\leq 1$ and $x_i+x_n\geq 1$, then the left hand side is $0$.
    \end{itemize}
    \item Suppose $i+1\in T^*$. If $i$ is odd, then $i+1$ is even and there is no arc $(i+1,i+2)$. This forces $j=i$, so that
      the cycle $C$ is the triangle shown in Figure~\ref{fig:case_3_possible_starts_2}(b). Then
      \[\cost(C)=\min\{x_i,1-x_n\}-\min\{1,x_i+x_{i+1}\}+x_n+\min\{x_{i+1},1-x_n\}.\]
      To see that this cannot be negative we verify
      \[\min\{x_i,1-x_n\}+\min\{x_{i+1},1-x_n\}\geq \min\{1,x_i+x_{i+1}\}-x_n.\]
      Consequently, we can assume that $i$ is even. In order to rule out $i\in T^*+1$, see
      Figure~\ref{fig:case_3_possible_starts_2}(c), we verify
      \[-\max\{0,x_i-x_n\}+\max\{1,x_i+x_{i+1}\}-x_n\geq\min\{x_{i+1},1-x_n\}.\]
      \begin{itemize}
      \item If $x_i\leq x_n$ then the left hand side is $\max\{1,x_i+x_{i+1}\}-x_n\geq 1-x_n$.
      \item If $x_i>x_n$ then the left hand side is $\max\{1,x_i+x_{i+1}\}-x_i\geq x_{i+1}$. \qedhere 
      \end{itemize}
    \end{enumerate}    
  \end{proof}
  \begin{lemma}\label{lem:case_3_possible_ends_1}
    $j+1\not\in T^*$
  \end{lemma}
  \begin{proof}
    Suppose $j+1\in T^*$. We distinguish the three cases indicated in Figure~\ref{fig:case_3_possible_ends_1}.
    \begin{figure}[htb]
      \begin{minipage}{.32\linewidth}
        \begin{tikzpicture}
          \node[draw,fill,circle,outer sep=2pt,inner sep=1pt,label={left:{\small $O$}}] (O) at (0,0) {};
          \node[draw,fill,circle,outer sep=2pt,inner sep=1pt,label={right:{\small $j$}}]
          (i) at (0:3) {};
          \node[draw,fill,circle,outer sep=2pt,inner sep=1pt,label={right:{\small $j+1$}}]
          (j) at (40:3) {};
          \draw[thick,<-] (O) to node[below] {{\small $M_j$}} (i);
          \draw[thick,->] (i) to node[right] {{\small $M'_j$}} (j);
          \draw[thick,<-] (O) to node[left] {{\small $-m_{j+1}$}} (j);
          \node at (1.5,-.8) {{\small (a) $j$ is even}};
        \end{tikzpicture}
      \end{minipage}
      \begin{minipage}{.32\linewidth}
        \begin{tikzpicture}
          \node[draw,fill,circle,outer sep=2pt,inner sep=1pt,label={left:{\small $O$}}] (O) at (0,0) {};
          \node[draw,fill,circle,outer sep=2pt,inner sep=1pt,label={right:{\small $j$}}]
          (i) at (0:3) {};
          \node[draw,fill,circle,outer sep=2pt,inner sep=1pt,label={right:{\small $j+1$}}]
          (j) at (40:3) {};
          \draw[thick,<-] (O) to node[below] {{\small $-m_j$}} (i);
          \draw[thick,->] (i) to node[right] {{\small $-m'_j$}} (j);
          \draw[thick,<-] (O) to node[left] {{\small $M_{j+1}$}} (j);
          \node at (1.5,-.8) {{\small (b) $j$ is odd and $j\in T^*+1$}};
        \end{tikzpicture}
      \end{minipage}
      \begin{minipage}{.32\linewidth}
        \begin{tikzpicture}
          \node[draw,fill,circle,outer sep=2pt,inner sep=1pt,label={left:{\small $O$}}] (O) at (0,0) {};
          \node[draw,fill,circle,outer sep=2pt,inner sep=1pt,label={right:{\small $j$}}]
          (i) at (0:3) {};
          \node[draw,fill,circle,outer sep=2pt,inner sep=1pt,label={right:{\small $j+1$}}]
          (j) at (40:3) {};
          \draw[thick,<-] (O) to node[below] {{\small $-M_j$}} (i);
          \draw[thick,->] (i) to node[right] {{\small $-m'_j$}} (j);
          \draw[thick,<-] (O) to node[left] {{\small $M_{j+1}$}} (j);
          \node at (1.5,-.8) {{\small (c) $j$ is odd and $j\not\in T^*+1$}};
        \end{tikzpicture}
      \end{minipage}
      \caption{The three cases in the proof of Lemma~\ref{lem:case_3_possible_ends_1}.}
      \label{fig:case_3_possible_ends_1}
    \end{figure}
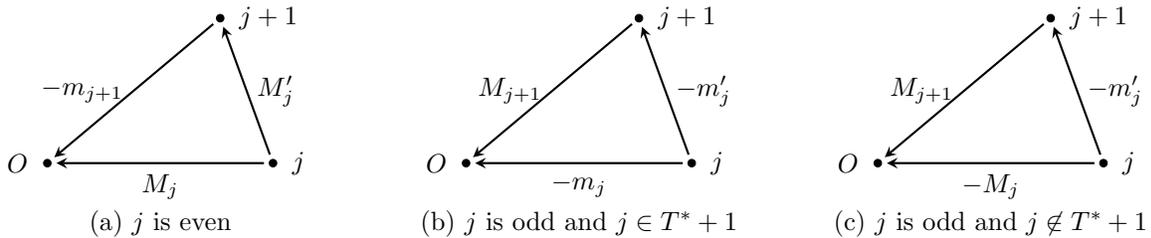
    \begin{description}
    \item[Case (a)] $j$ is even. We need to show that
      \[\max\{1,x_j+x_{j+1}\}-x_n-\max\{0,x_{j+1}-x_n\}\geq\min\{x_j,1-x_n\}.\]
      \begin{itemize}
      \item If $x_{j+1}\leq x_n$ then the left hand side is $\max\{1,x_j+x_{j+1}\}-x_n\geq 1-x_n$.
      \item If $x_{j+1}> x_n$ then the left hand side is $\max\{1,x_j+x_{j+1}\}-x_{j+1}\geq x_j$.
      \end{itemize}
    \item[Case (b)] $j$ is odd and $j\in T^*+1$. We need to show that
      \[-\min\{1,x_j+x_{j+1}\}+x_n+\min\{x_{j+1},1-x_n\}\geq-\max\{0,x_j-x_n\}=\min\{0,x_n-x_j\}.\]
      \begin{itemize}
      \item If $x_{j+1}+x_n\leq 1$ then the left hand side is $x_n+x_{j+1}-\min\{1,x_j+x_{j+1}\}\geq
        x_n-x_j$.
      \item If $x_{j+1}+x_n> 1$ then the left hand side is $1-\min\{1,x_j+x_{j+1}\}\geq 0$. 
      \end{itemize}
    \item[Case (c)] $j$ is odd and $j\not\in T^*+1$. We need to show that
      \[-\min\{1,x_j+x_{j+1}\}+x_n+\min\{x_{j+1},1-x_n\}\geq-\min\{x_j,1-x_n\}=\max\{-x_j,x_n-1\}.\]
      \begin{itemize}
      \item If $x_{j+1}+x_n\leq 1$ then the left hand side is $x_n+x_{j+1}-\min\{1,x_j+x_{j+1}\}\geq
        x_n-x_j$.
      \item If $x_{j+1}+x_n> 1$ then the left hand side is $1-\min\{1,x_j+x_{j+1}\}\geq 0$.\qedhere
      \end{itemize}
    \end{description}
  \end{proof}
  \begin{lemma}\label{lem:case_3_possible_ends_2}\hfill
    \begin{enumerate}
    \item If $j\in T^*$ then $j$ is odd and either $x_{j+1}+\max\{x_j,x_n\}<1$ or
      $x_{j+1}+\min\{x_j,x_n\}>1$.
    \item If $j\not\in T^*$ then $j$ is even. 
    \end{enumerate}
  \end{lemma}
  \begin{proof}
    \begin{figure}[htb]
      \centering
      \begin{minipage}{.32\linewidth}
        \begin{tikzpicture}
          \node[draw,fill,circle,outer sep=2pt,inner sep=1pt,label={left:{\small $O$}}] (O) at (0,0) {};
          \node[draw,fill,circle,outer sep=2pt,inner sep=1pt,label={right:{\small $j$}}]
          (i) at (0:3) {};
          \node[draw,fill,circle,outer sep=2pt,inner sep=1pt,label={right:{\small $j+1$}}]
          (j) at (40:3) {};
          \draw[thick,<-] (O) to node[below] {{\small $-m_j$}} (i);
          \draw[thick,->] (i) to node[right] {{\small $-M'_j$}} (j);
          \draw[thick,<-] (O) to node[left] {{\small $M_{j+1}$}} (j);
          \node at (1.5,-.8) {{\small (a) $j\in T^*$ (and odd)}};
        \end{tikzpicture}
      \end{minipage}
      \begin{minipage}{.32\linewidth}
        \begin{tikzpicture}
          \node[draw,fill,circle,outer sep=2pt,inner sep=1pt,label={left:{\small $O$}}] (O) at (0,0) {};
          \node[draw,fill,circle,outer sep=2pt,inner sep=1pt,label={right:{\small $j$}}]
          (i) at (0:3) {};
          \node[draw,fill,circle,outer sep=2pt,inner sep=1pt,label={right:{\small $j+1$}}]
          (j) at (40:3) {};
          \draw[thick,<-] (O) to node[below] {{\small $-m_j$}} (i);
          \draw[thick,->] (i) to node[right] {{\small $-m'_j$}} (j);
          \draw[thick,<-] (O) to node[left] {{\small $M_{j+1}$}} (j);
          \node at (1.5,-.8) {{\small (b) $j$ is odd and $j\in T^*+1$}};
        \end{tikzpicture}
      \end{minipage}
      \begin{minipage}{.32\linewidth}
        \begin{tikzpicture}
          \node[draw,fill,circle,outer sep=2pt,inner sep=1pt,label={left:{\small $O$}}] (O) at (0,0) {};
          \node[draw,fill,circle,outer sep=2pt,inner sep=1pt,label={right:{\small $j$}}]
          (i) at (0:3) {};
          \node[draw,fill,circle,outer sep=2pt,inner sep=1pt,label={right:{\small $j+1$}}]
          (j) at (40:3) {};
          \draw[thick,<-] (O) to node[below] {{\small $-M_j$}} (i);
          \draw[thick,->] (i) to node[right] {{\small $-m'_j$}} (j);
          \draw[thick,<-] (O) to node[left] {{\small $M_{j+1}$}} (j);
          \node at (1.5,-.8) {{\small (c) $j$ is odd and $j\not\in T^*\cup(T^*+1)$}};
        \end{tikzpicture}
      \end{minipage}
      \caption{Illustrations for the proof of Lemma~\ref{lem:case_3_possible_ends_2}}
      \label{fig:case_3_possible_ends_2}
    \end{figure}
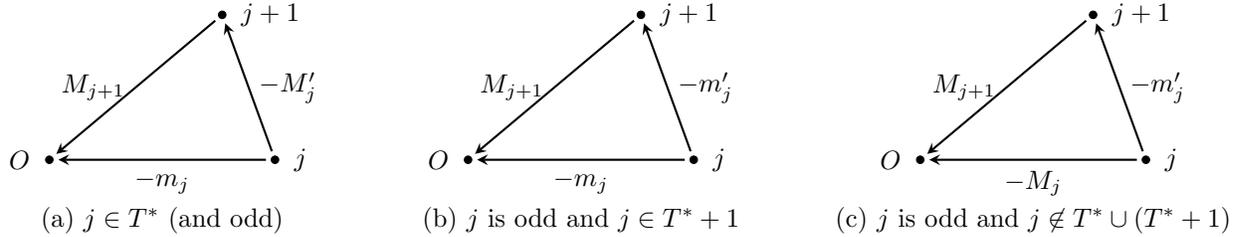
    \begin{enumerate}
    \item If $j\in T^*$ then $j$ must be odd for the arc $(j,j+1)$ to exist. Now suppose
      $x_{j+1}+\max\{x_j,x_n\}\geq 1$ and $x_{j+1}+\min\{x_j,x_n\}\leq 1$. We need to show that (see
      Figure~\ref{fig:case_3_possible_ends_2}(a))
      \[-\max\{1,x_j+x_{j+1}\}+x_n+\min\{x_{j+1},1-x_n\}\geq-\max\{0,x_j-x_n\}=\min\{0,x_n-x_j\}.\]
      \begin{itemize}
      \item If $x_{j+1}+x_j\geq 1$ and $x_{j+1}+x_n\leq 1$ then the left hand side is $x_n-x_j$.
      \item If $x_{j+1}+x_j\leq 1$ and $x_{j+1}+x_n\geq 1$ then the left hand side is $0$. 
      \end{itemize}
    \item Suppose $j\not\in T^*$ and $j$ is odd (see Figure~\ref{fig:case_3_possible_ends_2}(b) and
      (c)). Since $-m_j\geq-M_j$ it is sufficient to verify 
      \[-\min\{1,x_j+x_{j+1}\}+x_n+\min\{x_{j+1},1-x_n\}\geq -\max\{0,x_j-x_n\}=\min\{0,x_n-x_j\}.\]
      \begin{itemize}
      \item If $x_{j+1}+x_n\leq 1$ then the left hand side is $x_n+x_{j+1}-\min\{1,x_j+x_{j+1}\}\geq
        x_n-x_j$.
      \item If $x_{j+1}+x_n> 1$ then the left hand side is $1-\min\{1,x_j+x_{j+1}\}\geq 0$.\qedhere
      \end{itemize}
    \end{enumerate}
  \end{proof}
  By Lemmas~\ref{lem:case_3_possible_starts_1},~\ref{lem:case_3_possible_starts_2},~\ref{lem:case_3_possible_ends_1}
  and \ref{lem:case_3_possible_ends_2} we are left with the following possibilities (illustrated in Figure ):
  \begin{enumerate}
  \item $i$ odd and $j$ odd. Then $i\in T^*$ and $j\in T^*$. 
  \item $i$ odd and $j$ even. Then $i\in T^*$ and $j\not\in T^*$, $j+1\not\in T^*$.
  \item $i$ even and $j$ odd. Then $i+1\in T^*$, $i-1\not\in T^*$ and $j\in T^*$. 
  \item $i$ even and $j$ even. Then $i+1\in T^*$, $i-1\not\in T^*$ and $j\not\in T^*$, $j+1\not\in T^*$.
  \end{enumerate}
  \begin{figure}[htb]
    \begin{minipage}{.49\linewidth}
      \begin{tikzpicture}
          \node[draw,fill,circle,outer sep=2pt,inner sep=1pt,label={left:{\small $O$}}] (O) at (0,0) {};
          \node[draw,fill,circle,outer sep=2pt,inner sep=1pt,label={right:{\small $i$}}]
          (i) at (0:3) {};
          \node[draw,fill,circle,outer sep=2pt,inner sep=1pt,label={30:{\small $i+1$}}]
          (i1) at (30:3) {};
          \node[draw,fill,circle,outer sep=2pt,inner sep=1pt,label={90:{\small $j$}}]
          (j) at (90:3) {};
          \node[draw,fill,circle,outer sep=2pt,inner sep=1pt,label={120:{\small $j+1$}}]
          (j1) at (120:3) {};
          \draw[thick,->] (O) to node[below] {{\small $M_i$}} (i);
          \draw[thick,->] (i) to node[right] {{\small $-M'_i$}} (i1);
          \draw[thick,->] (j) to node[above] {{\small $-M'_j$}} (j1);
          \draw[thick,->] (j1) to node[left] {{\small $M_{j+1}$}} (O);
          \node at (1.5,-.8) {{\small (a) $i$ and $j$ odd.}};
          \draw [thin,dashed,domain=30:90] plot ({3*cos(\x)}, {3*sin(\x)});
        \end{tikzpicture}
      \end{minipage}\hfill
      \begin{minipage}{.49\linewidth}
      \begin{tikzpicture}
          \node[draw,fill,circle,outer sep=2pt,inner sep=1pt,label={left:{\small $O$}}] (O) at (0,0) {};
          \node[draw,fill,circle,outer sep=2pt,inner sep=1pt,label={right:{\small $i$}}]
          (i) at (0:3) {};
          \node[draw,fill,circle,outer sep=2pt,inner sep=1pt,label={30:{\small $i+1$}}]
          (i1) at (30:3) {};
          \node[draw,fill,circle,outer sep=2pt,inner sep=1pt,label={90:{\small $j$}}]
          (j) at (90:3) {};
          \node[draw,fill,circle,outer sep=2pt,inner sep=1pt,label={120:{\small $j+1$}}]
          (j1) at (120:3) {};
          \draw[thick,->] (O) to node[below] {{\small $M_i$}} (i);
          \draw[thick,->] (i) to node[right] {{\small $-M'_i$}} (i1);
          \draw[thick,->] (j) to node[above] {{\small $M'_j$}} (j1);
          \draw[thick,->] (j1) to node[left] {{\small $-M_{j+1}$}} (O);
          \node at (1.5,-.8) {{\small (b) $i$ odd and $j$ even.}};
          \draw [thin,dashed,domain=30:90] plot ({3*cos(\x)}, {3*sin(\x)});
        \end{tikzpicture}
      \end{minipage}
      \begin{minipage}{.49\linewidth}
      \begin{tikzpicture}
          \node[draw,fill,circle,outer sep=2pt,inner sep=1pt,label={left:{\small $O$}}] (O) at (0,0) {};
          \node[draw,fill,circle,outer sep=2pt,inner sep=1pt,label={right:{\small $i$}}]
          (i) at (0:3) {};
          \node[draw,fill,circle,outer sep=2pt,inner sep=1pt,label={30:{\small $i+1$}}]
          (i1) at (30:3) {};
          \node[draw,fill,circle,outer sep=2pt,inner sep=1pt,label={90:{\small $j$}}]
          (j) at (90:3) {};
          \node[draw,fill,circle,outer sep=2pt,inner sep=1pt,label={120:{\small $j+1$}}]
          (j1) at (120:3) {};
          \draw[thick,->] (O) to node[below] {{\small $-M_i$}} (i);
          \draw[thick,->] (i) to node[right] {{\small $M'_i$}} (i1);
          \draw[thick,->] (j) to node[above] {{\small $-M'_j$}} (j1);
          \draw[thick,->] (j1) to node[left] {{\small $M_{j+1}$}} (O);
          \node at (1.5,-.8) {{\small (c) $i$ even and $j$ odd.}};
          \draw [thin,dashed,domain=30:90] plot ({3*cos(\x)}, {3*sin(\x)});
        \end{tikzpicture}
      \end{minipage}\hfill
      \begin{minipage}{.49\linewidth}
      \begin{tikzpicture}
          \node[draw,fill,circle,outer sep=2pt,inner sep=1pt,label={left:{\small $O$}}] (O) at (0,0) {};
          \node[draw,fill,circle,outer sep=2pt,inner sep=1pt,label={right:{\small $i$}}]
          (i) at (0:3) {};
          \node[draw,fill,circle,outer sep=2pt,inner sep=1pt,label={30:{\small $i+1$}}]
          (i1) at (30:3) {};
          \node[draw,fill,circle,outer sep=2pt,inner sep=1pt,label={90:{\small $j$}}]
          (j) at (90:3) {};
          \node[draw,fill,circle,outer sep=2pt,inner sep=1pt,label={120:{\small $j+1$}}]
          (j1) at (120:3) {};
          \draw[thick,->] (O) to node[below] {{\small $-M_i$}} (i);
          \draw[thick,->] (i) to node[right] {{\small $M'_i$}} (i1);
          \draw[thick,->] (j) to node[above] {{\small $M'_j$}} (j1);
          \draw[thick,->] (j1) to node[left] {{\small $-M_{j+1}$}} (O);
          \node at (1.5,-.8) {{\small (d) $i$ and $j$ even.}};
          \draw [thin,dashed,domain=30:90] plot ({3*cos(\x)}, {3*sin(\x)});
        \end{tikzpicture}
    \end{minipage}
    \caption{The four types of cycles in Case 3.}
    \label{fig:types_case_3}
  \end{figure}
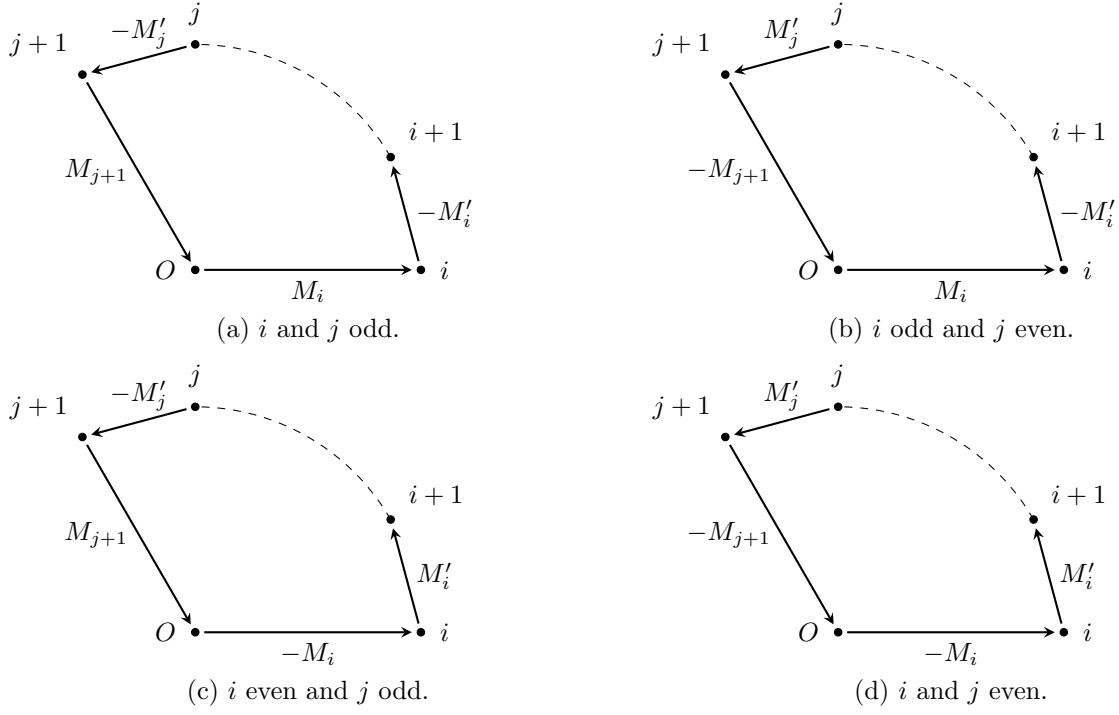
  In every case, the elements of $T^*\cap[i,j]$ are the odd $k\in[i,j]$ with $1\leq
  x_k+x_{k+1}+x_n\leq 2$, and the set
  \[T = \left(T^*\setminus[i,j])\cup\{k\in[i,j]\,:\,k\text{ even  and }1\leq x_k+x_{k+1}+x_n\leq
      2\right\}\]
  is feasible, that is, $T\cap(T+1)=\emptyset$. We will write down $\cost(C)$ and
  $\phi(T)-\phi(T^*)$, and then see that $\cost(C)<0$ implies $\phi(T)>\phi(T^*)$, as required.
  Similar to what we did in the dicussion of Case 1, we introduce the following notation:
  \begin{align*}
    S_0 &= \{k\in[i,j]\,:\,k\text{ odd and }x_k+x_{k+1}+x_n<1\}, & S_2 &= \{k\in[i,j]\,:\,k\text{
                                                                         odd and
                                                                         }x_k+x_{k+1}+x_n>2\},\\
    S'_0 &= \{k\in[i,j]\,:\,k\text{ even and }x_k+x_{k+1}+x_n<1\}, & S'_2 &= \{k\in[i,j]\,:\,k\text{
                                                                         even and
                                                                         }x_k+x_{k+1}+x_n>2\}.    
  \end{align*}
  In all cases the cost is a slight modification of
  \[\cost_0(C)=\sum_{\text{even }k\in[i,j]}\max\{1,x_k+x_{k+1}\}-\sum_{\text{odd
      }k\in[i,j]}\max\{1,x_k+x_{k+1}\}+\sum_{k\in S_0}(1-x_k-x_{k+1})+\sum_{k\in
      S_2}(x_k+x_{k+1}-1).\]
  More precisely,
  \begin{align}
    \cost(C) &= \cost_0(C)+\min\{x_i,1-x_n\}+\min\{x_{j+1},1-x_n\}+x_n &&\text{if both $i$ and $j$
                                                                          are
                                                                          odd},\label{eq:cost_oo}\\
    \cost(C) &= \cost_0(C)+\min\{x_i,1-x_n\}-\min\{x_{j+1},1-x_n\} &&\text{if $i$ is odd and $j$
                                                                          is
                                                                      even},\label{eq:cost_oe}\\
    \cost(C) &= \cost_0(C)-\min\{x_i,1-x_n\}+\min\{x_{j+1},1-x_n\} &&\text{if $i$ is even and $j$
                                                                          is
                                                                      odd},\label{eq:cost_eo}\\
    \cost(C) &= \cost_0(C)-\min\{x_i,1-x_n\}-\min\{x_{j+1},1-x_n\}-x_n &&\text{if both $i$ and $j$
                                                                          are even}.\label{eq:cost_ee}
  \end{align}
  Next we will show that
   \begin{multline}\label{eq:punchline}
    \phi(T)-\phi(T^*)=-\cost(C)+\sum_{k\in S'_0}(1-x_k-x_{k+1}-x_n)+\sum_{k\in
      S'_2}(x_k-x_{k+1}-x_n-2)\\
    +\lvert S_0\rvert x_n+\lvert S_2\rvert(1-x_n),
  \end{multline}
  and this will conclude the proof in Case 3. Let $\Phi_0$ be the common part of $\phi(T^*)$ and $\phi(T)$, that is everything beyond the
  interval $[i,j+1]$. If $i$ and $j$ are both odd, then 
     \begin{multline*}
    \phi(T^*)=\Phi_0+\sum_{k\in[i,j+1]}x_k+x_n(j-i+2)/2-(j-i+1)+\sum_{\text{even
      }k\in[i,j]}\max\{1,x_k+x_{k+1}\}\\
    +\sum_{k\in S_0}(1-x_k-x_{k+1}-x_n)+\sum_{k\in S_2}(x_k+x_{k+1}+x_n-2),
  \end{multline*}
  and
   \begin{multline*}
     \phi(T)=\Phi_0+\max\{0,x_i+x_n-1\}+\max\{0,x_{j+1}+x_n-1\}+\sum_{k\in[i+1,j]}x_k+x_n(j-i)/2-(j-i+1)\\
     +\sum_{\text{odd }k\in[i,j]}\max\{1,x_k+x_{k+1}\}+\sum_{k\in S'_0}(1-x_k-x_{k+1}-x_n)+\sum_{k\in S'_2}(x_k+x_{k+1}+x_n-2).
  \end{multline*}
  Using $\max\{0,x_k+x_n-1\}-x_k=-\min\{x_k,1-x_n\}$ for $k=i$ and $k=j+1$, we obtain 
  \begin{multline*}
    \phi(T)-\phi(T^*)=\sum_{\text{odd }k\in[i,j]}\max\{1,x_k+x_{k+1}\}-\sum_{\text{even
      }k\in[i,j]}\max\{1,x_k+x_{k+1}\}\\
    -\min\{x_i,1-x_n\}-\min\{x_{j+1},1-x_n\}-x_n +\sum_{k\in S'_0}(1-x_k-x_{k+1}-x_n)+\sum_{k\in
      S'_2}(x_k-x_{k+1}-x_n-2)\\
    -\sum_{k\in S_0}(1-x_k-x_{k+1}-x_n)-\sum_{k\in S_2}(x_k-x_{k+1}-x_n-2),
  \end{multline*}
  and with~(\ref{eq:cost_oo}) we obtain~(\ref{eq:punchline}). If $i$ is even and $j$ is odd, then
  the same calculations lead to
  \begin{multline*}
    \phi(T)-\phi(T^*)=\sum_{\text{odd }k\in[i,j]}\max\{1,x_k+x_{k+1}\}-\sum_{\text{even
      }k\in[i,j]}\max\{1,x_k+x_{k+1}\}\\
    -\min\{x_i,1-x_n\}+\min\{x_{j+1},1-x_n\} +\sum_{k\in S'_0}(1-x_k-x_{k+1}-x_n)+\sum_{k\in
      S'_2}(x_k-x_{k+1}-x_n-2)\\
    -\sum_{k\in S_0}(1-x_k-x_{k+1}-x_n)-\sum_{k\in S_2}(x_k-x_{k+1}-x_n-2),
  \end{multline*}
  and with~(\ref{eq:cost_oe}) we obtain~(\ref{eq:punchline}). If $i$ is even and $j$ is odd, then
  \begin{multline*}
    \phi(T)-\phi(T^*)=\sum_{\text{odd }k\in[i,j]}\max\{1,x_k+x_{k+1}\}-\sum_{\text{even
      }k\in[i,j]}\max\{1,x_k+x_{k+1}\}\\
    +\min\{x_i,1-x_n\}-\min\{x_{j+1},1-x_n\}-x_n +\sum_{k\in S'_0}(1-x_k-x_{k+1}-x_n)+\sum_{k\in
      S'_2}(x_k-x_{k+1}-x_n-2)\\
    -\sum_{k\in S_0}(1-x_k-x_{k+1}-x_n)-\sum_{k\in S_2}(x_k-x_{k+1}-x_n-2),
  \end{multline*}
  and with~(\ref{eq:cost_eo}) we obtain~(\ref{eq:punchline}). Finally, if both $i$ and $j$ are even,
  then
  \begin{multline*}
    \phi(T)-\phi(T^*)=\sum_{\text{odd }k\in[i,j]}\max\{1,x_k+x_{k+1}\}-\sum_{\text{even
      }k\in[i,j]}\max\{1,x_k+x_{k+1}\}\\
    +\min\{x_i,1-x_n\}+\min\{x_{j+1},1-x_n\}+x_n +\sum_{k\in S'_0}(1-x_k-x_{k+1}-x_n)+\sum_{k\in
      S'_2}(x_k-x_{k+1}-x_n-2)\\
    -\sum_{k\in S_0}(1-x_k-x_{k+1}-x_n)-\sum_{k\in S_2}(x_k-x_{k+1}-x_n-2),
  \end{multline*}
  and with~(\ref{eq:cost_ee}) we obtain~(\ref{eq:punchline}).
\end{appendix}

\end{document}